\documentclass{amsart}

\usepackage{mathtools, 
            amssymb, 
            hyperref, 
            graphicx,
}
\usepackage[letterpaper, left=1in, right=1in]{geometry}  
\usepackage[shortlabels]{enumitem}

\usepackage{tikz}
\usetikzlibrary{arrows, positioning, math, quotes, calc, decorations.pathreplacing}

\makeatletter
\renewcommand{\sectionautorefname}{\S\@gobble}
\renewcommand{\subsectionautorefname}{\S\@gobble}
\makeatother

\newcommand{\FF}{\mathbb{F}}
\newcommand{\QQ}{\mathbb{Q}}
\newcommand{\RR}{\mathbb{R}}
\newcommand{\ZZ}{\mathbb{Z}}
\newcommand{\cB}{\mathcal{B}}
\newcommand{\cP}{\mathcal{P}}
\newcommand{\cT}{\mathcal{T}}
\newcommand{\bx}{\mathbf{x}}
\newcommand{\ones}{\mathbf{1}}

\DeclareMathOperator{\conv}{conv}
\DeclareMathOperator{\rank}{rank}

\newcommand{\paren}[1]{\left( #1 \right)}

\newcommand{\set}[1]{\left\{ #1 \right\}}

\newcommand{\floor}[1]{\left\lfloor #1 \right\rfloor}
\newcommand{\ceil}[1]{\left\lceil #1 \right\rceil}

\usepackage{amsthm}
\usepackage{thmtools}
\theoremstyle{plain}
\newtheorem{theorem}{Theorem}[section]
\newtheorem{lemma}[theorem]{Lemma}
\newtheorem{definition}[theorem]{Definition}
\newtheorem{proposition}[theorem]{Proposition}
\newtheorem{corollary}[theorem]{Corollary}
\newtheorem{question}{Question}

\theoremstyle{remark}

\declaretheoremstyle[
  spaceabove=\medskipamount, spacebelow=\medskipamount,
  headfont=\normalfont\scshape,
  notefont=\mdseries, notebraces={(}{)},
  bodyfont=\normalfont,
  postheadspace=1em,
  qed=\qedsymbol
]{examplestyle}
\declaretheorem[style=examplestyle, sibling=theorem]{example}

\usepackage[backend=biber, style=alphabetic, url=false, doi=false, isbn=false, maxnames=50]{biblatex}
\title{Uniformly Weighted Graphical Designs}
\author{Zawad Chowdhury}
\author{Rekha R. Thomas}
\address{Department of Mathematics, University of Washington, Seattle WA 98195}
\email{zawadx@uw.edu, rrthomas@uw.edu}

\begin{document}

\begin{abstract}
A graphical design is a subset of vertices of a graph, along with a weight for each chosen vertex, that can perfectly average chosen subspaces of functions on the graph. 
A design is uniformly weighted if all the weights are equal, and several well-known combinatorial objects such as orthogonal arrays, combinatorial block designs and $t$-wise permutations are uniformly weighted graphical designs.
While one might expect to see uniformly weighted designs in structured graphs, they do not always exist. In this paper we characterize the existence of uniformly weighted graphical designs, and use our result to provide several families of graphs that have, and do not have, such designs. Our results offer a polyhedral view of the structures that control the existence and cardinalities of these designs. 
In particular, we characterize all uniformly weighted designs of threshold graphs, and provide a geometric proof of the duality of linear codes and linear orthogonal arrays.
We also provide a novel construction for graphs whose Laplacian characteristic polynomials are almost irreducible, to produce families without uniformly weighted designs.

\end{abstract} 

\maketitle

\section{Introduction}

\subsection{Graphical Designs} 
Let $G = ([n],E)$ be a connected undirected graph with vertex set $[n] := \{1,\ldots, n\}$ and edge set $E$.
A subset of vertices $W \subset [n]$, with weights $a_w \in \RR$ for each $w \in W$, \textbf{averages} a function $f: [n] \to \RR$ on $G$ if the global average of $f$ on $G$ equals the weighted sum of $f$ at the vertices in $W$, i.e.,
\begin{align} \label{eq:global-average-equals-wted-sum}
    \frac{1}{n} \sum_{v \in [n]} f(v) = \sum_{w \in W} a_w f(w).
\end{align}
We will be concerned with subsets $W$ that can average entire subspaces of functions on $G$ with uniform weights ($a_w = a$ for all $w \in W$), providing combinatorial quadrature rules on graphs. 

Functions on $G$ can be organized via the Laplacian of $G$, the $n \times n$ matrix $L$ defined entrywise by
\[
    L_{ij} = \begin{cases}
        \deg(i) & \text{ if } i=j \ (\text{$\deg(i)$ is the degree of vertex $i$}), \\
        -1 & \text{ if } (i, j) \in E,\\
        0 & \text{ otherwise.}
    \end{cases}
\]
The Laplacian $L$ is positive semidefinite with real eigenvalues. 
Let the distinct eigenvalues of $L$ be $0 = \lambda_0 < \lambda_1 < \cdots < \lambda_k$, the multiplicity of $\lambda_i$ be $d_i$, and the $d_i$-dimensional eigenspace of $\lambda_i$ be $\Lambda_i$. 
The smallest eigenvalue is $\lambda_0 = 0$ with a one-dimensional eigenspace $\Lambda_0$ spanned by the all-ones vector $\ones := (1,1,\ldots,1) \in \RR^n$; we refer to this eigenvalue and eigenspace as ``trivial''. The second smallest eigenvalue $\lambda_1$ is positive because $G$ is connected. 
Note that our indexing is non-traditional; usually $\lambda_1 = 0$ is the first eigenvalue (and not treated as ``trivial''), and the second eigenvalue measuring connectivity is denoted by $\lambda_2$.

\begin{definition} \label{def:graphical-designs}
    A {\bf graphical design} of $G=([n],E)$, with \textbf{strength $t$}, is a subset of vertices $W \subset [n]$, with weights $a_w \in \mathbb{R}$ for each $w \in W$, such that \autoref{eq:global-average-equals-wted-sum} holds for all $f \in \Lambda_0 \oplus \Lambda_1 \oplus \cdots \oplus \Lambda_{t}$.
    \begin{enumerate}
        \item The design $W$ is {\bf positively weighted} if $a_w >  0$ for all $w \in W$. 
        \item The design $W$ is {\bf uniformly weighted} if 
        $a_w = a > 0$ for all $w \in W$. 
    \end{enumerate}
\end{definition}

Graphical designs were introduced by Steinerberger \cite{steinerberger2020designs} and developed further in a number of recent papers \cite{golubev2020extremal, babecki2021codes, babecki2022galeduality, babecki-shiroma, zhu2023bch, steinerbergerthomas2025randomwalks, chowdhury2025combstructures}. 

By \autoref{def:graphical-designs}, a graphical design of strength $t$ can average all elements in the first $t+1$ eigenspaces of $L$, where the eigenspaces have been ordered by increasing magnitude of their eigenvalues. This particular ordering (called \textbf{Laplacian order}) is natural, as it orders functions on $G$ by an intuitive notion of ``smoothness''. There are other choices for ordering eigenspaces as in \cite{golubev2020extremal, babecki2021codes, babecki2022galeduality, chowdhury2025combstructures}.
Two easy facts about designs are that no proper subset of vertices can average all eigenspaces of $G$, and that every subset of vertices will average $\Lambda_0$. Therefore, the strengths of interest lie in the range $1 \leq t \leq k-1$, and the nontrivial requirement on a design of strength $t$ is that it must average the first $t$ nontrivial eigenspaces of $L$.

\subsection{Uniformly Weighted Designs}
In principle, the weights $a_w$ on a design $W$ can be any real numbers. However, the most interesting situation is when they are positively weighted.
Babecki and Thomas \cite{babecki2022galeduality} prove a structure theorem for positively weighted designs, which shows that {\em every} graph $G$ admits positively weighted designs of {\em every} strength. On the other hand, uniformly weighted designs are rare and often fail to exist even in highly structured graphs. This prompts the question:
\begin{question}\label{qn:exist}
    (Existence) When does a graph $G$ have a uniformly weighted design of a given strength? 
\end{question}
\autoref{thm:unif-wted-two-valued}, which is the main result of this paper, answers \autoref{qn:exist}.
A graph $G$ has a uniformly weighted design of strength $t$ if and only if a certain point configuration (called an {\em eigenconfiguration} and obtained from eigenvectors that the design does not average) takes exactly two distinct values on some linear functional. This provides us with a geometric tool to certify the existence of uniformly weighted designs, using which we construct many examples of graphs with and without uniformly weighted designs. In the reverse direction, designs give us polyhedral information about eigenconfigurations and their convex hulls, called {\em eigenpolytopes}, and we elaborate on this at several points in the paper. 

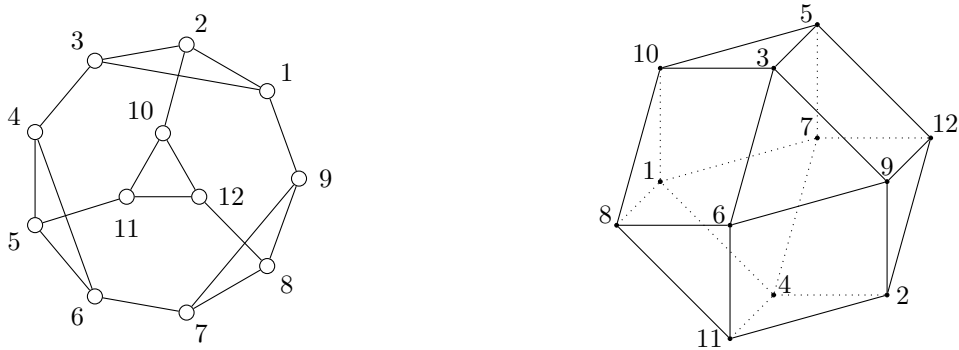
\begin{figure}[!h]
    \begin{minipage}{0.48\textwidth}
        \centering
        \begin{tikzpicture}[scale=0.6]
            \tikzstyle{every node}=[circle, fill=white, draw=black, inner sep=2pt]

            \foreach \a in {1,2,...,9}{
                \node (o\a) at (\a*360/9: 3) [label=\a*360/9:$\a$]{};
            }
            \node (i1) at (0, 1) [label=above left:$10$]{};
            \node (i2) at (-0.8, -0.4) [label=below:$11$]{};
            \node (i3) at (0.8, -0.4) [label=right:$12$]{};

            \foreach \a[evaluate={\b=int(\a + 1)}] in {1,2,...,8}{
                \draw (o\a) -- (o\b);
            }
            \draw (o9) -- (o1);
            \foreach \a[evaluate={\c=int(\a + 2)}] in {1,4,7}{
                \draw (o\a) -- (o\c);
            }
            \draw (i1) -- (i2) -- (i3) -- (i1);
            \foreach \a[evaluate={\d=int(3*\a - 1)}] in {1,2,3}{
                \draw (i\a) -- (o\d);
            }
        \end{tikzpicture}
    \end{minipage}
    \begin{minipage}{0.48\textwidth}
        \centering
        \begin{tikzpicture}[scale=1.5]
            \tikzstyle{every node}=[circle, fill=black, draw=black, inner sep=0.5pt]

            \node (1) at (-1, 0, 0) [label=above left:$1$]{};
            \node (2) at (1, -1, 0) [label=right:$2$]{};
            \node (3) at (0, 1, 0) [label=above left:$3$]{};
            \node (4) at (0, -1, 0) [label=above right:$4$]{};
            \node (5) at (0, 1, -1) [label=above left:$5$]{};
            \node (6) at (0, 0, 1) [label=above left:$6$]{};
            \node (7) at (0, 0, -1) [label=above left:$7$]{};
            \node (8) at (-1, 0, 1) [label=above left:$8$]{};
            \node (9) at (1, 0, 0) [label=above:$9$]{};
            \node (10) at (-1, 1, 0) [label=above left:$10$]{};
            \node (11) at (0, -1, 1) [label=left:$11$]{};
            \node (12) at (1, 0, -1) [label=above right:$12$]{};

            \draw (2) -- (12); 
            \draw (2) -- (9); 
            \draw[dotted] (2) -- (4); 
            \draw (2) -- (11); 
            \draw (12) -- (9); 
            \draw[dotted] (12) -- (7); 
            \draw (12) -- (5); 
            \draw (9) -- (6); 
            \draw (9) -- (3); 
            \draw[dotted] (4) -- (11); 
            \draw[dotted] (4) -- (7); 
            \draw[dotted] (4) -- (1); 
            \draw (11) -- (6); 
            \draw (11) -- (8); 
            \draw[dotted] (7) -- (5); 
            \draw[dotted] (7) -- (1); 
            \draw (6) -- (3); 
            \draw (6) -- (8); 
            \draw (5) -- (3); 
            \draw (5) -- (10); 
            \draw (3) -- (10); 
            \draw[dotted] (1) -- (8); 
            \draw[dotted] (1) -- (10); 
            \draw (8) -- (10); 
            
        \end{tikzpicture}
    \end{minipage}
    \caption{For the truncated tetrahedron graph (left), the eigenpolytope for designs averaging all but the last eigenspace is the cuboctahedron (right).}
    \label{fig:eigenconfig-example}
\end{figure}
The entire vertex set $V$ is certainly a uniformly weighted design of every strength, but not useful for the intended goal of sampling functions with a small subset of the vertices. Therefore it is interesting to try to find uniformly weighted designs of small cardinality. For positively weighted designs, we have a bound coming from the dimension of the eigenconfiguration.

\begin{definition}
    A positively weighted design $W$ of strength $t$ is {\bf minimal} if there is no proper subset of $W$ that is a positively weighted design of strength $t$.
\end{definition}
\begin{theorem} \cite[Theorem~3.14]{babecki2022galeduality} \label{thm:dim-bound}
Every minimal positively weighted design of strength $t$ has cardinality at most $s_t := \sum_{i=0}^{t} d_i$. 
\end{theorem}
In other words, the bound $s_t$ on the size of a minimal positively weighted design of strength $t$ is the sum of dimensions of the $t+1$ eigenspaces that the design must average. We say that a design $W$ of strength $t$ meets the \textbf{dimension bound} if its cardinality, $|W|$, is at most $s_t$.

\begin{question}\label{qn:size}
    (Cardinality) If $G$ has uniformly weighted designs of strength $t$, is there one that meets the dimension bound?
\end{question}

While the dimension bound is guaranteed if the uniformly weighted design is minimal, it is possible that none of the minimal positively weighted designs of a particular strength are uniformly weighted. \autoref{ex:triangular-honeycomb} shows a graph in which there are uniformly weighted designs of a particular strength, but none that meets the dimension bound, highlighting the nontriviality of \autoref{qn:size}. In the next two subsections we highlight the historical interest in the cardinality question from independent research communities. 

\subsection{Connections to Combinatorial Structures}
There is a long history in combinatorics of deciding whether highly structured combinatorial objects can have small cardinality. We highlight three such objects. An {\em orthogonal array} of strength $t$ with $q$ levels is a subset of $\{0,1, \dots, q-1\}^n$ such that all elements of $\{0,1, \dots, q-1\}^t$ appear an equal number of times when the array is restricted to any collection of $t$ or less coordinates. 
A $t$-$(n,k,\lambda)$-{\em combinatorial block design} is a 
collection of $k$-element subsets of $[n]$ such that every $t$-element set in $[n]$ occurs in exactly $\lambda$ elements of the design. 
Finally, a $t${\em-wise permutation} is a subset of permutations on $[n]$ with a uniform action on any $t$-tuple of elements in $[n]$. 
The existence of combinatorial block designs for all parameters satisfying some easy-to-obtain divisibility constraints was a long-standing conjecture, resolved recently by Keevash \cite{keevash2024existencedesigns}.
Kuperberg, Lovett and Peled \cite{kuperberg2017combstructures} provided a unified framework for the existence of the three structures. Through probabilistic arguments, they show that for fixed $t$, the above structures exist with cardinality at most $(cn)^{ct}$ for some universal constant $c > 0$.

A different unified framework for these structures is that they are all uniformly weighted graphical designs for appropriately defined graphs \cite{chowdhury2025combstructures}. Orthogonal arrays are designs of the \emph{Hamming graph}, combinatorial block designs are designs of the {\em Johnson graph}, and $t$-wise permutations are designs of the transposition graph on the symmetric group $\mathfrak{S}_n$. 
Moreover, for fixed $t$, the sum of dimensions of eigenspaces averaged by these graphical designs is $O(n^t)$ in all three cases. This coincides with the bounds in \cite{kuperberg2017combstructures}, with a cleaner constant in the exponent than the one derived by probabilistic methods. 
Our main result, \autoref{thm:unif-wted-two-valued}, shows that there is a deterministic polyhedral reason for why the bounds in \cite{kuperberg2017combstructures} hold, continuing the unifying theme in \cite{chowdhury2025combstructures}.

\subsection{Connections to Spherical Designs}
Graphical designs are discrete analogs of {\em spherical designs}, which are well-known quadrature rules on the sphere \cite{sobolev1962cubature, goethals1977sphericaldesigns}. Let $S^n$ be the unit sphere in $\RR^{n+1}$ with normalized Lebesgue measure $\mu_n$. A set of points $x_1, \ldots, x_N \in S^n$ is called a {\em spherical} $t$-{\em design} if 
\[\int_{S^n} p(x) d \mu_n(x) = \frac{1}{N} \sum_{i=1}^N p(x_i)\]
for all homogeneous polynomials $p$ in $n+1$ variables of degree at most $t$. Spherical $t$-designs were shown to exist for all $n,t \in \ZZ_{\ge 0}$ by Seymour and Zaslavsky \cite{seymourzaslavsky1984spericaldesign}. 
It was a long-standing open problem, known as the {\em Korevaar-Meyers conjecture}, that the smallest spherical $t$-designs in $S^n$ have cardinality $O(t^n)$.
This conjecture was famously resolved in the positive by Bondarenko, Radchenko and Viazovska \cite{bondarenko2013spherical}. Note that $t^n$ is the order of the dimension of the space of homogeneous polynomials in $n+1$ variables of degree at most $t$ and hence the upper bound conjectured by Korevaar and Meyers is precisely the dimension bound. 
Also note that in this case the dimension bound was $O(t^n)$, while in the previous subsection it was $O(n^t)$. This is due to the convention of what is fixed to obtain the asymptotics. For combinatorial structures, we must have $1 \le t \le n$, so it makes sense to fix $t$ and find the bound as a function of $n$. Since there are spherical designs for all $t \in \ZZ_{\ge 0}$ given any $n$, the convention is to fix $n$ and obtain bounds as a function of $t$.

Although uniformly weighted graphical designs are the discrete analogs of spherical designs, not all graphs have uniformly weighted designs of all strengths (\autoref{ex:icosahedron}), and some graphs have no uniformly weighted design of any strength (\autoref{sec:algebraic-methods-no-designs}). 
This paper offers a complete geometric picture of when uniformly weighted graphical designs exist, and the polyhedral structures that control their size. 

\subsection{Organization of This Paper}
In \autoref{sec:weights-of-designs} we establish our main tool, \autoref{thm:unif-wted-two-valued}, which characterizes when a graph $G$ admits a uniformly weighted design of strength $t$. We illustrate the tool on a graph in which all minimal positively weighted designs of every strength are uniformly weighted (\autoref{ex:C6+K4bar}); all of them meet the dimension bound. On the other hand, \autoref{ex:triangular-honeycomb} shows a graph that admits uniformly weighted designs of strength one, but none of them are minimal and none of them meet the dimension bound. 
The problem of deciding whether a graph admits a uniformly weighted design of strength $t$ is a binary integer program. \autoref{prop:subset-sum} shows that this problem, for strength one designs in a weighted graph, is equivalent to the NP-complete multidimensional subset sum problem.

Point configurations that take exactly two values with respect to a linear functional are at the heart of the answer to \autoref{qn:exist}. In \autoref{sec:2-valued} we establish several geometric results about such configurations. 
These polyhedral results are employed in \autoref{ex:icosahedron} to show that the icosahedral graph admits uniformly weighted designs at some but not all strengths. They also produce nontrivial designs on the binary hypercube when the strength $ t \le \floor{2n/3}$ (\autoref{ex:cube-smallsize-design}). 

While the previous sections exhibited individual graphs with desired properties, in \autoref{sec:all-strengths} we provide three families of graphs that have uniformly weighted designs of all strengths. The first is a family of {\em complete tripartite graphs} (\autoref{ex:tripartite}) in which all minimal designs of all strengths are uniformly weighted. The second is the family of {\em threshold graphs}, for which a single vertex is a design of all strengths, but not all minimal designs are uniformly weighted. We then revisit hypercube graphs, where again, there are uniformly weighted designs of all strengths that meet the dimension bound, but not all minimal designs are uniformly weighted. Our methods for the hypercube graph provide a new geometric proof of the duality of linear codes and linear orthogonal arrays (\autoref{prop:geom-proof-duality}).

We conclude in \autoref{sec:algebraic-methods-no-designs} by constructing infinite families of graphs in which there are no uniformly weighted designs of any strength. This relies on a result relating the algebra of the Laplacian eigenvalues with the possible weights of graphical designs (\autoref{thm:rational-wts-thm}). Using this result, we establish a criterion that can produce infinite families of graphs without uniformly weighted designs (\autoref{thm:no-unif-design-inf-families}), as well as two concrete examples of such infinite families (\autoref{ex:monorail-graphs}, \autoref{ex:houseboat-graphs}).

\subsection*{Acknowledgements} We thank Stefan Steinerberger for many helpful conversations and initial computational explorations. Some examples in this paper were computed using Mathematica with coding help from Claude; others were computed in SageMath. We also used Claude for literature review and TikZ help. 
The writing and proofs in this paper did not use LLMs.

\section{Weights of designs} \label{sec:weights-of-designs}

In this section we characterize when uniformly weighted graphical designs exist, extending the results in \cite{babecki2022galeduality} on the existence of positively weighted designs. 
We begin by recalling some simple observations about \autoref{def:graphical-designs} from \cite{babecki2022galeduality}. In what follows we only consider strengths $1 \le t \le k-1$, which are the interesting cases for \autoref{qn:exist}. 
First, for a design to average all elements in an eigenspace of $L$, it suffices for it to average a basis of eigenvectors in that eigenspace. Secondly, if $\varphi \in \Lambda_i$ for some $i > 0$, then $\ones^\top \varphi = 0$ because the eigenspaces of $L$ are mutually orthogonal. This means that $W$, with weights $\{a_w\}$, averages $\varphi$ if and only if $0 = \sum_{w \in W} a_w \varphi(w)$. In particular, a uniformly weighted $W$ (with weight $a$) can average $\varphi$ if and only if $0 = a \sum_{w \in W} \varphi(w)$, which is equivalent to $0 = \sum_{w \in W} \varphi(w)$. Hence we may assume that $a=1$, and focus on the sets $W$.
Lastly, $W$ is a uniformly weighted design of strength $t$ if and only if $\overline{W} = [n] \setminus W$ is also a uniformly weighted design of strength $t$. Indeed, $W$ is a uniformly weighted design of strength $t$ if and only if $0 = \sum_{w \in W} \varphi(w)$ for all $\varphi$ in a basis of each of $\Lambda_1 \ldots, \Lambda_{t}$. Since any such $\varphi$ is orthogonal to $\ones$,
\[0 = \ones^\top \varphi = \sum_{w \in W} \varphi(w) + \sum_{w \in \overline{W}} \varphi(w).\]
Therefore, $\sum_{w \in W} \varphi(w) = 0$ if and only if 
$\sum_{w \in \overline{W}} \varphi(w) = 0$.

We now begin to answer \autoref{qn:exist}. 
Let $U \in \RR^{n \times n}$ be a matrix with $k+1$ blocks of rows such that the $i$th block has $d_i$ rows that form a basis of the eigenspace $\Lambda_i$ of $L$. We assume that the blocks are arranged in Laplacian order so that the top row is the vector $\ones$.
Suppose we wish to find $W \subset [n]$ with weights $\{a_w\}$ that averages the first $t$ nontrivial eigenspaces of $L$. Denote by $A$ the submatrix of $U$ consisting of blocks corresponding to $\lambda_1, \ldots, \lambda_{t}$ (the eigenvectors that $W$ must average), and by $B$ the submatrix of row blocks corresponding to $\lambda_{t+1}, \ldots, \lambda_k$, so that 
\begin{align} \label{eq:U}
U = \begin{bmatrix} 1 & \cdots & 1  \\  -  & A & -  \\ - & B & -   \end{bmatrix}.
\end{align}
Let $\alpha := \sum_{i=1}^{t} d_i$ be the number of rows in $A$ and $\beta := \sum_{i={t+1}}^k d_i = n - \alpha - 1$ be the number of rows in $B$. Since we are interested in $1 \leq t \leq k-1$, the matrix $A$ has at least $d_1 > 0$ rows and $B$ has at least $d_k > 0$ rows. Lastly, let the columns of $B$ be $b_1, \ldots, b_n \in \RR^{\beta}$. 

\begin{definition}
    Call the multiset $\cB_t := \{b_1, \ldots, b_n \} \subset \RR^{\beta}$ the $t$-th {\bf eigenconfiguration} of $G$, and the $\beta$-dimensional polytope $P_t := \conv(\cB_t)$, the convex hull of $\cB_t$, the $t$-th {\bf eigenpolytope} of $G$.
\end{definition}

The main theorem in \cite{babecki2022galeduality} proves that the minimal positively weighted designs of $G$ of strength $t$ are in bijection with the facets (codimension-one faces) of the eigenpolytope $P_t$.

\begin{theorem} \cite[]{babecki2022galeduality} \label{thm:gale-duality-bijection} A set of vertices $W \subset [n]$ is a minimal positively weighted design of $G$, of strength $t$, if and only if there is a facet of the eigenpolytope $P_t$ containing precisely $\{b_i \in \cB_t \,:\, i \in \overline{W}\}$.
\end{theorem}

\begin{example}
    The truncated tetrahedron graph (\autoref{fig:eigenconfig-example}, left) has five Laplacian eigenspaces. The last eigenspace $\Lambda_4$ is three-dimensional, and the columns of a basis of $\Lambda_4$ produce the eigenconfiguration $\cB_3$. The corresponding eigenpolytope $P_3$ is the cuboctahedron (\autoref{fig:eigenconfig-example}, right). It has fourteen facets, eight of which are triangles and six of which are quadrilaterals. Therefore by \autoref{thm:gale-duality-bijection}, there are fourteen minimal positively weighted designs of strength 3 in the graph, eight with $12-3 = 9$ vertices in their support and six with $12 - 4 = 8$ vertices. 
\end{example}

\autoref{thm:dim-bound} follows immediately from \autoref{thm:gale-duality-bijection}: every facet of $P_t$ contains at least $n-s_t$ points (the dimension of the polytope), which means that the set of points in $\cB_t$ that do not lie on a facet of $P_t$ has cardinality at most $s_t$.
The tools used to prove \autoref{thm:gale-duality-bijection} can also extract the weights for a minimal design of strength $t$. While this was not developed in \cite{babecki2022galeduality}, we do so now to answer \autoref{qn:exist}.

\begin{definition} \label{def:2-valued conf} 
 A non-coplanar point configuration $\cP = \{p_1, \ldots, p_r \} \subset \RR^d$ is \textbf{$2$-valued} if there is a linear functional $a^\top x$ that takes exactly two values on the points in $\cP$. 
\end{definition}

\begin{theorem} \label{thm:unif-wted-two-valued} 
 A graph $G$ has a uniformly weighted design of strength $t$ if and only if the $t$-th eigenconfiguration $\cB_t$ is $2$-valued.
\end{theorem}

\begin{proof}
Consider the matrix $U$ from \autoref{eq:U}. 
A subset $W \subset [n]$ is a uniformly weighted design of $G$ of strength $t$ if and only if $A e_W = 0$, where $e_W = \sum_{i \in W} e_i$ and $e_i$ is the $i$-th standard basis vector in $\RR^n$. From the orthogonality of eigenspaces of $L$ and the definition of $U$, we have that 
\[A \begin{bmatrix} 1  & b_1^\top \\ 1 & b_2^\top \\ \vdots & \vdots \\ 1 & b_n^\top  \end{bmatrix} = 0, \] 
and that the columns of the matrix on the right form a basis of the kernel of $A$. Therefore, $e_W$ lies in the kernel of $A$ if and only if there is a unique real vector $\bar{a} := (a_0, a) \in \mathbb{R}^{1 + \beta}$ such that 
\[a_0 + a^\top b_i = 1\,\,\forall i \in W, \qquad
a_0 + a^\top b_i = 0 \,\,\forall i \not \in W. \]
Equivalently, the eigenconfiguration $\cB_t$ is contained in two parallel hyperplanes
and hence, $\cB_t$ is $2$-valued with respect to $a^\top x$. 
\end{proof}

As a corollary to Theorems~\ref{thm:gale-duality-bijection} and \ref{thm:unif-wted-two-valued}, we have the following statement. 

\begin{corollary} \label{cor:minimal-unif-designs}
    A set $W \subset [n]$ is a minimal uniformly weighted design of strength $t$ in $G$ if and only if there is a facet $F$ of the eigenpolytope $P_t$ containing all $b_i$ such that $i \not \in W$ and a parallel translate of the supporting hyperplane of $F$ containing all   $b_i$ such that $i \in W$.
\end{corollary}

The proof of \autoref{thm:unif-wted-two-valued} can be modified easily to show that $W$ is a design of strength $t$ with $\ell$ distinct positive weights $\delta_1 < \delta_2 < \ldots < \delta_\ell$ if and only if the $t$-th eigenconfiguration $\cB_t$ is contained in $\ell+1$ parallel affine hyperplanes of the form $a_0 + a^\top x = 0, a_0 + a^\top x = \delta_1, \ldots, a_0 + a^\top x = \delta_\ell$. 

The eigenconfiguration $\cB_t$ and hence $P_t = \conv(\cB_t)$ depend on our choice of eigenbases. However, since two different choices of $\cB_t$ differ by an invertible linear transformation $Q$, a different eigenconfiguration would be $\cB_t' = \{Qb_1, \ldots, Qb_n\}$. We can check that $\cB_t$ lies on two parallel hyperplanes with normal vector $a$ if and only if $\cB_t'$ lies on two parallel hyperplanes with normal vector $(Q^{-1})^\top a$. Therefore, \autoref{thm:unif-wted-two-valued} does not depend on the choice of eigenbases of $\Lambda_1, \ldots, \Lambda_k$.

\begin{example} \label{ex:C6+K4bar}
    Consider the regular graph $C_6 + \overline{K_4}$ of degree $6$, constructed by joining every vertex of the cycle $C_6$ to the complement of the complete graph $K_4$ (in other words, to $4$ isolated vertices). The Laplacian of this graph has  eigenvalues $0, 5^{(2)}, 6^{(3)}, 7^{(2)}, 8, 10$, where the exponent denotes the multiplicity of each eigenvalue. In this case $k=5$, and the graph could have nontrivial uniformly weighted designs of strengths $t=1,2,3,4$. The graph, alongside its $U$-matrix as in \autoref{eq:U}, is provided in \autoref{fig:C6K4bar}.
    We now check for uniformly weighted designs of various strengths in this graph using \autoref{thm:unif-wted-two-valued}.    
    \begin{figure}[!h]
        \begin{minipage}{0.40\textwidth}
            \centering
            \begin{tikzpicture}[scale=0.5]
                \tikzstyle{every node}=[circle, fill=white, draw=black, inner sep=2pt]
                \tikzstyle{labelnode}=[rectangle, draw=none, fill=none, inner sep=0pt]
                \foreach \a in {1,2,...,6}{
                    \node (o\a) at (\a*60+30: 4) [label=\a*60+30:$\a$]{};
                }
                \foreach \a in {1,2,...,4}{
                    \node (i\a) at (\a*90+45: 1.5) {};
                }
                \node[labelnode, above=0.8em of i1] {$7$};
                \node[labelnode, below=0.8em of i2] {$8$};
                \node[labelnode, below=0.8em of i3] {$9$};
                \node[labelnode, above=0.8em of i4] {$10$};
                
                \foreach \a[evaluate={\b=int(\a + 1)}] in {1,2,...,5}{
                    \draw (o\a) -- (o\b);
                }
                \draw (o6) -- (o1);
                \foreach \a in {1,2,...,6}{
                    \foreach \b in {1,2,...,4}{
                        \draw (o\a) -- (i\b);
                    }
                }
            \end{tikzpicture}
        \end{minipage}
        \begin{minipage}{0.56\textwidth}
            \centering
            $U = 
            \left( \begin{array}{rrrrrrrrrr}
            1& 1& 1& 1& 1& 1& 1& 1& 1& 1\\ 
            \hline
            -1& -1& 0& 1& 1& 0& 0& 0& 0& 0\\ 
            1&0& -1& -1& 0& 1& 0& 0& 0& 0\\ 
            \hline
            0& 0& 0& 0& 0& 0& -1& 1& 0& 0\\ 
            0& 0& 0& 0& 0& 0& -1& 0& 1& 0\\ 
            0& 0& 0& 0& 0& 0& -1& 0& 0& 1\\
            \hline
            -1& 1& 0& -1& 1& 0& 0& 0& 0& 0\\ -1& 0& 1& -1& 0& 1& 0& 0& 0& 0\\ \hline
            -1& 1& -1& 1& -1& 1& 0& 0& 0& 0\\ \hline
            -2& -2& -2& -2& -2& -2& 3& 3& 3& 3
            \end{array} \right)$
        \end{minipage}
        \caption{Left: the graph $C_6 + \overline{K_4}$. Right: its $U$-matrix.}
        \label{fig:C6K4bar}
    \end{figure}
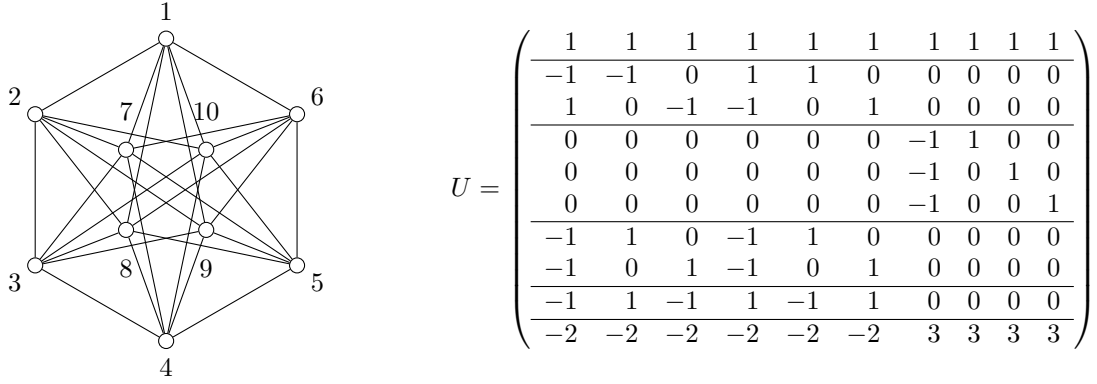
    \begin{enumerate}
        \item For $t=4$, the eigenconfiguration 
        $\cB_4 = \{-2, -2, -2, -2, -2, -2, 3, 3, 3, 3 \}$ is $2$-valued with respect to the linear functional $x$. 
        Therefore, $W = \{7,8,9,10\}$ supported on $\overline{K_4}$ is a uniformly weighted minimal design of strength $4$ in $G$. The complement $\overline{W} = \{1,2,3,4,5,6\}$ supported on $C_6$ is also a uniformly weighted minimal design of strength $4$. Both meet the dimension bound $s_4 = 9$. One can check that $W$ and $\overline{W}$ both average the first $9$ rows of $U$ with uniform weight $1$. Since $P_4 = [-2,3]$ has only two facets, we have found all minimal positively weighted designs of strength $4$.
    
        \item For $t=3$, $\cB_3$ consists of the $10$ columns from the last two rows of $U$, and $P_3 = \conv(\cB_3)$ is the triangle with vertices $(-1,-2), (1,-2), (0,3)$ that are repeated 3, 3 and 4 times respectively. The configuration $\cB_3$ is $2$-valued with respect to the supporting hyperplane of every side of the triangle. Therefore all minimal designs are uniformly weighted and meet the dimension bound $s_3 = 8$:
        $$ \{7,8,9,10\}, \{1,3,5\}, \{2,4,6\}.$$
    
        \item For $t=2$, the eigenpolytope $P_2$ is $4$-dimensional with $7$ vertices and $6$ facets. It is a pyramid over a triangular prism with the apex of the pyramid at $(0,0,0,3)$, repeated $4$ times. The eigenconfiguration $\cB_2$ is $2$-valued with respect to $a^\top x$ for every facet normal $a$ of $P_2$. The minimal designs are:  
        $$\{1,4\}, \{2,5\}, \{3,6\}, \{1,3,5\}, \{2,4,6\}, \{7,8,9,10\}.$$
        They are all uniformly weighted, meet the dimension bound of $s_2=6$, and are illustrated in \autoref{fig:C6K4bar-designs}.
    
        \item The last strength of interest is $t=1$ for which the eigenpolytope is $7$-dimensional and is the join of a triangular prism formed by the first six points and a tetrahedron formed by the last four points. Therefore, it has $6+4=10$ vertices and $5+4=9$ facets.
        All minimal designs are again uniformly weighted and meet the dimension bound of $s_1 = 3$:
        $$\{7\}, \{8\}, \{9\}, \{10\}, \{1,4\}, \{2,5\}, \{3,6\}, \{1,3,5\}, \{2,4,6\}.$$
        The singleton designs come from the facets joining the prism to each of the $4$ facets of the tetrahedron. The two-element designs come from joining the tetrahedron to the three quadrilateral facets of the prism. And finally, the three-element designs come from joining the tetrahedron to the two triangular facets of the prism.
    \end{enumerate}
    \begin{figure}[!h]
        \begin{minipage}{0.48\textwidth}
            \begin{tikzpicture}[yscale=1.3]
                \tikzstyle{repeatnode}=[circle, fill=white, draw=black, inner sep=1.5pt]
                \node (1to6) at (2, 3) {$\{1,2,3,4,5,6\}$};
                \node (7to10) at (5.8, 3) {$\{7,8,9,10\}$};
                \node (135) at (0, 2) {$\{1,3,5\}$};
                \node (14) at (1, 1) {$\{1,4\}$};
                \node (25) at (2, 1) {$\{2,5\}$};
                \node (36) at (3, 1) {$\{3,6\}$};
                \node (246) at (4, 2) {$\{2,4,6\}$};
                \node[repeatnode] (r2) at ((5.8, 2) {};
                \node[repeatnode] (r1) at ((5.8, 1) {};
                \node[repeatnode] (l1a) at (0, 1) {};
                \node[repeatnode] (l1b) at (4, 1) {};
                \node (7) at (4.8, 0) {$\{7\}$};
                \node (8) at (5.4, 0) {$\{8\}$};
                \node (9) at (6, 0) {$\{9\}$};
                \node (10) at (6.7, 0) {$\{10\}$};
                \node[repeatnode] (l0a) at (0, 0) {};
                \node[repeatnode] (l0c) at (1, 0) {};
                \node[repeatnode] (l0d) at (2, 0) {};
                \node[repeatnode] (l0e) at (3, 0) {};
                \node[repeatnode] (l0b) at (4, 0) {};
    
                \draw (1to6) -- (135) -- (l1a) -- (l0a);
                \draw (1to6) -- (246) -- (l1b) -- (l0b);
                \draw (1to6) -- (14) -- (l0c);
                \draw (1to6) -- (25) -- (l0d);
                \draw (1to6) -- (36) -- (l0e);
                \draw (7to10) -- (r2) -- (r1);
                \draw (8) -- (r1) -- (7);
                \draw (9) -- (r1) -- (10);
            \end{tikzpicture}
        \end{minipage}
        \hspace{0.04\textwidth}
        \begin{minipage}{0.44\textwidth}
        \centering
        \begin{minipage}{0.32\textwidth}
            \centering
            \begin{tikzpicture}[scale=0.25]
                \tikzstyle{every node}=[circle, fill=white, draw=black, inner sep=1.5pt]
                \tikzstyle{labelnode}=[rectangle, draw=none, fill=none, inner sep=0pt]
                \foreach \a in {1,2,...,6}{
                    \node (o\a) at (\a*60+30: 4) {};
                }
                \foreach \a in {1,2,...,4}{
                    \node (i\a) at (\a*90+45: 1.5) {};
                }
    
                \node[fill=black] at (o1) {};
                \node[fill=black] at (o4) {};
                
                \foreach \a[evaluate={\b=int(\a + 1)}] in {1,2,...,5}{
                    \draw (o\a) -- (o\b);
                }
                \draw (o6) -- (o1);
                \foreach \a in {1,2,...,6}{
                    \foreach \b in {1,2,...,4}{
                        \draw (o\a) -- (i\b);
                    }
                }
            \end{tikzpicture}
        \end{minipage}
        \begin{minipage}{0.32\textwidth}
            \centering
            \begin{tikzpicture}[scale=0.25]
                \tikzstyle{every node}=[circle, fill=white, draw=black, inner sep=1.5pt]
                \tikzstyle{labelnode}=[rectangle, draw=none, fill=none, inner sep=0pt]
                \foreach \a in {1,2,...,6}{
                    \node (o\a) at (\a*60+30: 4) {};
                }
                \foreach \a in {1,2,...,4}{
                    \node (i\a) at (\a*90+45: 1.5) {};
                }
    
                \node[fill=black] at (o2) {};
                \node[fill=black] at (o5) {};
    
                \foreach \a[evaluate={\b=int(\a + 1)}] in {1,2,...,5}{
                    \draw (o\a) -- (o\b);
                }
                \draw (o6) -- (o1);
                \foreach \a in {1,2,...,6}{
                    \foreach \b in {1,2,...,4}{
                        \draw (o\a) -- (i\b);
                    }
                }
            \end{tikzpicture}
        \end{minipage}
        \begin{minipage}{0.32\textwidth}
            \centering
            \begin{tikzpicture}[scale=0.25]
                \tikzstyle{every node}=[circle, fill=white, draw=black, inner sep=1.5pt]
                \tikzstyle{labelnode}=[rectangle, draw=none, fill=none, inner sep=0pt]
                \foreach \a in {1,2,...,6}{
                    \node (o\a) at (\a*60+30: 4) {};
                }
                \foreach \a in {1,2,...,4}{
                    \node (i\a) at (\a*90+45: 1.5) {};
                }
    
                \node[fill=black] at (o3) {};
                \node[fill=black] at (o6) {};

                \foreach \a[evaluate={\b=int(\a + 1)}] in {1,2,...,5}{
                    \draw (o\a) -- (o\b);
                }
                \draw (o6) -- (o1);
                \foreach \a in {1,2,...,6}{
                    \foreach \b in {1,2,...,4}{
                        \draw (o\a) -- (i\b);
                    }
                }
            \end{tikzpicture}
        \end{minipage}
        
        \vspace{1em}
        \begin{minipage}{0.32\textwidth}
            \centering
            \begin{tikzpicture}[scale=0.25]
                \tikzstyle{every node}=[circle, fill=white, draw=black, inner sep=1.5pt]
                \tikzstyle{labelnode}=[rectangle, draw=none, fill=none, inner sep=0pt]
                \foreach \a in {1,2,...,6}{
                    \node (o\a) at (\a*60+30: 4) {};
                }
                \foreach \a in {1,2,...,4}{
                    \node (i\a) at (\a*90+45: 1.5) {};
                }
    
                \node[fill=black] at (o1) {};
                \node[fill=black] at (o3) {};
                \node[fill=black] at (o5) {};

                \foreach \a[evaluate={\b=int(\a + 1)}] in {1,2,...,5}{
                    \draw (o\a) -- (o\b);
                }
                \draw (o6) -- (o1);
                \foreach \a in {1,2,...,6}{
                    \foreach \b in {1,2,...,4}{
                        \draw (o\a) -- (i\b);
                    }
                }
            \end{tikzpicture}
        \end{minipage}
        \begin{minipage}{0.32\textwidth}
            \centering
            \begin{tikzpicture}[scale=0.25]
                \tikzstyle{every node}=[circle, fill=white, draw=black, inner sep=1.5pt]
                \tikzstyle{labelnode}=[rectangle, draw=none, fill=none, inner sep=0pt]
                \foreach \a in {1,2,...,6}{
                    \node (o\a) at (\a*60+30: 4) {};
                }
                \foreach \a in {1,2,...,4}{
                    \node (i\a) at (\a*90+45: 1.5) {};
                }
    
                \node[fill=black] at (o2) {};
                \node[fill=black] at (o4) {};
                \node[fill=black] at (o6) {};
                
                \foreach \a[evaluate={\b=int(\a + 1)}] in {1,2,...,5}{
                    \draw (o\a) -- (o\b);
                }
                \draw (o6) -- (o1);
                \foreach \a in {1,2,...,6}{
                    \foreach \b in {1,2,...,4}{
                        \draw (o\a) -- (i\b);
                    }
                }
            \end{tikzpicture}
        \end{minipage}
        \begin{minipage}{0.32\textwidth}
            \centering
            \begin{tikzpicture}[scale=0.25]
                \tikzstyle{every node}=[circle, fill=white, draw=black, inner sep=1.5pt]
                \tikzstyle{labelnode}=[rectangle, draw=none, fill=none, inner sep=0pt]
                \foreach \a in {1,2,...,6}{
                    \node (o\a) at (\a*60+30: 4) {};
                }
                \foreach \a in {1,2,...,4}{
                    \node (i\a) at (\a*90+45: 1.5) {};
                }
    
                \node[fill=black] at (i1) {};
                \node[fill=black] at (i2) {};
                \node[fill=black] at (i3) {};
                \node[fill=black] at (i4) {};
                
                \foreach \a[evaluate={\b=int(\a + 1)}] in {1,2,...,5}{
                    \draw (o\a) -- (o\b);
                }
                \draw (o6) -- (o1);
                \foreach \a in {1,2,...,6}{
                    \foreach \b in {1,2,...,4}{
                        \draw (o\a) -- (i\b);
                    }
                }
            \end{tikzpicture}
        \end{minipage}
        \end{minipage}
        \caption{ $G = C_6 + \overline{K_4}$: uniformly weighted designs in each level form a poset by inclusions (left). The uniformly weighted minimal designs of strength $2$ (right).} 
        \label{fig:C6K4bar-designs}
    \end{figure}
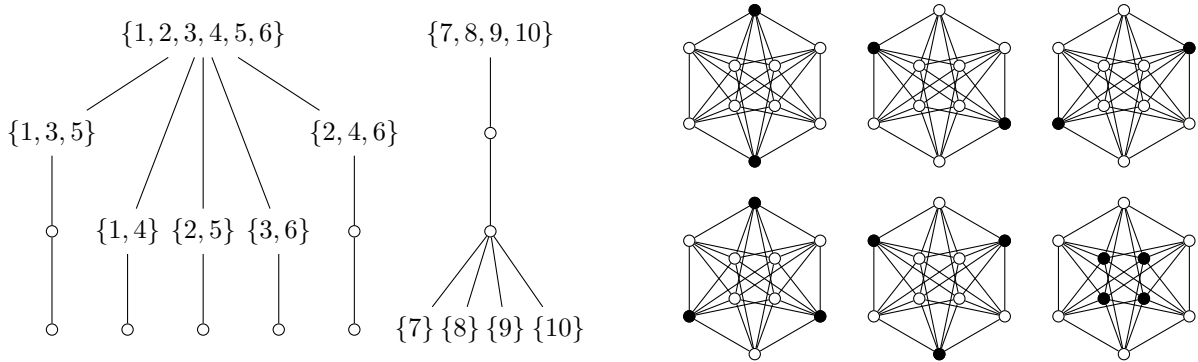
    Since all the minimal designs are uniformly weighted, they form a poset by inclusion (\autoref{fig:C6K4bar-designs}, left). For example, the sets $\{1,3,5\}$ and $\{2,4,6\}$ are themselves designs of strengths 1, 2 and 3. They are no longer designs of strength 4, but their union $\{1,2,3,4,5,6\}$ is a uniformly weighted design of strength 4.
\end{example}

While every minimal positively weighted design of every strength was uniformly weighted for $C_6 + \overline{K_4}$, this is far from true in general. In \autoref{sec:algebraic-methods-no-designs} we construct families of graphs with no uniformly weighted designs of any strength. An intermediate situation would be if $G$ has uniformly weighted designs of a given strength, but none of them are minimal. We provide such a graph in the following example.

\begin{example} \label{ex:triangular-honeycomb}
    Let $G$ be the {\em Triangular Honeycomb Queen $4$} graph.
    The Laplacian of $G$ has eigenvalues $0, 5^{(3)}, 6^{(2)}, 8^{(3)}, 9$. The graph and its $U$-matrix are shown in \autoref{fig:triangular-honeycomb-queen4}.
    
    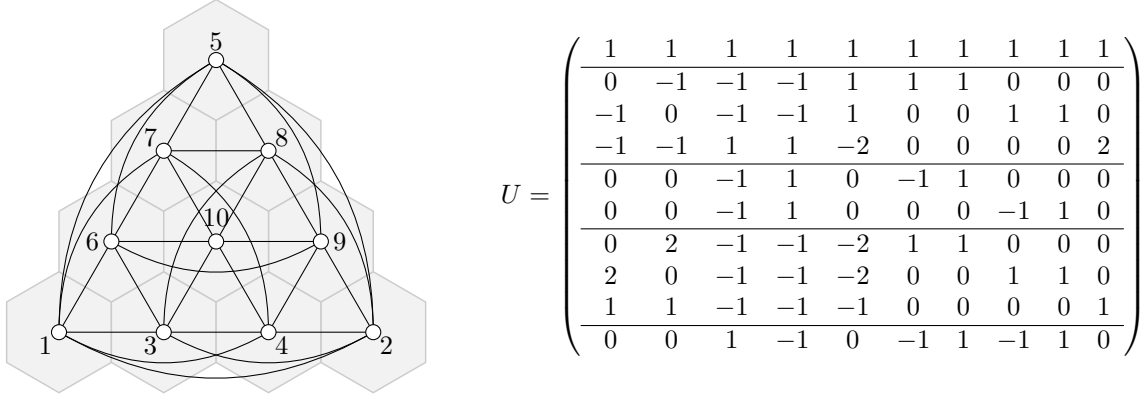
\begin{figure}[!h]
        \begin{minipage}{0.40\textwidth}
            \centering
            \begin{tikzpicture}[scale=0.8]
                \tikzstyle{every node}=[circle, fill=white, draw=black, inner sep=2pt, label distance=-0.3em]
                \tikzstyle{honey}=[fill=gray!10, draw=gray!40, line width=.55pt]

                \coordinate (10) at (0, 0); 
                \coordinate (9) at (0:1.732);
                \coordinate (8) at (60:1.732);
                \coordinate (7) at (120:1.732);
                \coordinate (6) at (180:1.732);
                \coordinate (3) at (240:1.732);
                \coordinate (4) at (300:1.732);
                \coordinate (2) at ($(9)+(-60:1.732)$);
                \coordinate (1) at ($(3)+(180:1.732)$);
                \coordinate (5) at ($(7)+(60:1.732)$);

                \path[honey] (1) ++(-90:1) -- ++(30:1) -- ++(90:1) -- ++(150:1) -- ++(-150:1) -- ++(-90:1) -- cycle;
                \path[honey] (2) ++(-90:1) -- ++(30:1) -- ++(90:1) -- ++(150:1) -- ++(-150:1) -- ++(-90:1) -- cycle;
                \path[honey] (3) ++(-90:1) -- ++(30:1) -- ++(90:1) -- ++(150:1) -- ++(-150:1) -- ++(-90:1) -- cycle;
                \path[honey] (4) ++(-90:1) -- ++(30:1) -- ++(90:1) -- ++(150:1) -- ++(-150:1) -- ++(-90:1) -- cycle;
                \path[honey] (5) ++(-90:1) -- ++(30:1) -- ++(90:1) -- ++(150:1) -- ++(-150:1) -- ++(-90:1) -- cycle;
                \path[honey] (6) ++(-90:1) -- ++(30:1) -- ++(90:1) -- ++(150:1) -- ++(-150:1) -- ++(-90:1) -- cycle;
                \path[honey] (7) ++(-90:1) -- ++(30:1) -- ++(90:1) -- ++(150:1) -- ++(-150:1) -- ++(-90:1) -- cycle;
                \path[honey] (8) ++(-90:1) -- ++(30:1) -- ++(90:1) -- ++(150:1) -- ++(-150:1) -- ++(-90:1) -- cycle;
                \path[honey] (9) ++(-90:1) -- ++(30:1) -- ++(90:1) -- ++(150:1) -- ++(-150:1) -- ++(-90:1) -- cycle;
                \path[honey] (10) ++(-90:1) -- ++(30:1) -- ++(90:1) -- ++(150:1) -- ++(-150:1) -- ++(-90:1) -- cycle;

                \draw (1) -- (6) -- (7) -- (5);
                \draw (3) -- (10) -- (8);
                \draw (4) -- (9);
                \draw (3) to[bend left=30] (8);
                \draw (1) to[bend left=30] (7);
                \draw (6) to[bend left=30] (5);
                \draw (1) to[bend left=30] (5);

                \draw (5) -- (8) -- (9) -- (2);
                \draw (7) -- (10) -- (4);
                \draw (6) -- (3);
                \draw (7) to[bend left=30] (4);
                \draw (5) to[bend left=30] (9);
                \draw (8) to[bend left=30] (2);
                \draw (5) to[bend left=30] (2);

                \draw (2) -- (4) -- (3) -- (1);
                \draw (9) -- (10) -- (6);
                \draw (8) -- (7);
                \draw (9) to[bend left=30] (6);
                \draw (2) to[bend left=30] (3);
                \draw (4) to[bend left=30] (1);
                \draw (2) to[bend left=30] (1);

                \node (1) at (1) [label=below left:$1$]{};
                \node (2) at (2) [label=below right:$2$]{};
                \node (3) at (3) [label=below left:$3$]{};
                \node (4) at (4) [label=below right:$4$]{};
                \node (5) at (5) [label=above:$5$]{};
                \node (6) at (6) [label=left:$6$]{};
                \node (7) at (7) [label=above left:$7$]{};
                \node (8) at (8) [label=above right:$8$]{};
                \node (9) at (9) [label=right:$9$]{};
                \node (10) at (10) [label=$10$]{};
            \end{tikzpicture}
        \end{minipage}
        \begin{minipage}{0.56\textwidth}
            \centering
            $U = 
            \left( \begin{array}{cccccccccc}
                 1 & 1 & 1 & 1 & 1 & 1 & 1 & 1 & 1 & 1 \\
                 \hline
                 0 & -1 & -1 & -1 & 1 & 1 & 1 & 0 & 0 & 0 \\
                 -1 & 0 & -1 & -1 & 1 & 0 & 0 & 1 & 1 & 0 \\
                 -1 & -1 & 1 & 1 & -2 & 0 & 0 & 0 & 0 & 2 \\
                 \hline
                 0 & 0 & -1 & 1 & 0 & -1 & 1 & 0 & 0 & 0 \\
                 0 & 0 & -1 & 1 & 0 & 0 & 0 & -1 & 1 & 0 \\
                 \hline
                 0 & 2 & -1 & -1 & -2 & 1 & 1 & 0 & 0 & 0 \\
                 2 & 0 & -1 & -1 & -2 & 0 & 0 & 1 & 1 & 0 \\
                 1 & 1 & -1 & -1 & -1 & 0 & 0 & 0 & 0 & 1 \\
                 \hline
                 0 & 0 & 1 & -1 & 0 & -1 & 1 & -1 & 1 & 0 \\
            \end{array} \right)$
        \end{minipage}
        \caption{Left: The Triangular Honeycomb Queen 4 graph. Right: its $U$-matrix.}
        \label{fig:triangular-honeycomb-queen4}
    \end{figure}

The eigenpolytope $P_1$ is $6$-dimensional, and all 10 points in $\cB_1$ are its vertices. The polytope has 31 facets, of which 25 are simplicial (with 6 vertices) and 6 are nonsimplicial (with 7 vertices). Therefore, the minimal positively weighted designs of strength one have cardinalities $4$ and $3$. They are
$$\begin{array}{ll}
\textup{weights} & \textup{designs}\\
\hline
(2,1,1,1) & \{3, 5, 6, 8\}, \{3, 5, 6, 9\}, \{3, 5, 7, 8\}, \{3, 5, 7, 9\}, \{4, 5, 6, 8\}, \{4, 5, 6, 9\}, \{4, 5, 7, 8\}, \{4, 5, 7, 9\};\\

(1,1,2,1) & \{2, 3, 6, 8\}, \{2, 3, 6, 9\}, \{2, 3, 7, 8\}, \{2, 3, 7, 9\}, \{2, 4, 6, 8\}, \{2, 4, 6, 9\}, \{2, 4, 7, 8\}, \{2, 4, 7, 9\};\\

(1,1,1,2) & \{1, 2, 5, 10\}, \{1, 3, 6, 8\}, \{1, 3, 6, 9\}, \{1, 3, 7, 8\}, \{1, 3, 7, 9\}, \{1, 4, 6, 8\}, \{1, 4, 6, 9\}, \{1, 4, 7, 8\}, \{1, 4, 7, 9\};\\

(2,2,1) &  \{1, 8, 10\}, \{1, 9, 10\}, \{2, 6, 10\}, \{2, 7, 10\}, \{3, 5, 10\}, \{4, 5, 10\}.
\end{array}$$

Note that none of these minimal designs are uniformly weighted. However, $G$ does have a uniformly weighted design of strength 1: the set $W = \{3,4,5,7,9\}$. One can check that $(0,0,1,1,1,0,1,0,1,0)$ is orthogonal to rows $2,3,4$ of $U$, which form a basis for $\Lambda_1$. This design  
$W$ is a positive linear combination of the minimal designs $\{3,5,7,9\}$ and $\{4,5,7,9\}$ each with weights $(2,1,1,1)$. Indeed, $$(0,0,1,1,1,0,1,0,1,0) = \frac{1}{2} (0,0,2,0,1,0,1,0,1,0) + \frac{1}{2} (0,0,0,2,1,0,1,0,1,0).$$
Note that $W$ violates the dimension bound as $|W| = 5 > 4 = s_1$. One can check that no $0/1$ vector supported on four or fewer elements is orthogonal to the eigenbasis of $\Lambda_1$. Therefore the graph $G$ has uniformly weighted designs of strength 1, but none that are minimal or meet the dimension bound.
\end{example}

It was hard to find small integral graphs that behave like \autoref{ex:triangular-honeycomb}. This raises an interesting question: if a generic graph has uniformly weighted designs at some strength, should at least one of them meet the dimension bound? Alternatively, is the Triangular Honeycomb Queen 4 actually characteristic of the generic behavior, and most graphs with uniformly weighted designs have none that meet the dimension bound?

Referring back to \autoref{eq:U}, $G$ has a uniformly weighted design of strength $t$ if and only if $A e_W = 0$ for some nonempty $W \subset [n]$. This is a {\em multidimensional subset sum problem}, which asks whether a proper nonempty subset of a given set of vectors (the columns of $A$ in our case) sums to a target vector (the zero vector in our case). The subset sum problem with integer, and hence rational, input is NP-complete \cite{garey1979computers}. 

\begin{proposition} \label{prop:subset-sum}
    The multidimensional subset sum problem with rational input is equivalent to the problem of deciding whether a positively weighted graph $G$ has a uniformly weighted design of strength one.
\end{proposition}

\begin{proof}
    Consider a subset sum problem of the form $Ae_W = 0$, where $A \in \QQ^{p \times n}$ with $p < n$. Let $S_A$ denote the row span of $A$. We can assume that $\rank(A)=p$, since if $Ae_W=0$ has a solution and the rows are dependent, then we can delete rows until $A$ has full rank without affecting the solution. We can also replace the rows of $A$ with any basis for $S_A$, since $Ae_W = 0$ has a solution if and only if $QAe_W=0$ has a solution for any invertible matrix $Q$. In particular, we can assume that the rows of $A$ are rational and pairwise orthogonal. If $\ones \in S_A$, then the subset sum problem has no nonzero solution. Therefore, we may assume that $S_A$ is orthogonal to $\ones$. Lastly, if $p = n-1$, then $Ae_W = 0$ if and only if $e_W$ is the unique element up to scaling in the nullspace of $A$, which is $\ones$ and thus not a valid solution. So we can also assume $p < n-1$. 

    By Lemma 3.1 in \cite{babecki-shiroma}, any rational matrix $U$ with the block structure in \autoref{eq:U}, i.e., the first row is $\ones$, the next block is $A$ as in the previous paragraph, and the final block is $B$ chosen so that all rows of $U$ are pairwise orthogonal, is the Laplacian of a connected graph $G$ with positive rational edge weights. 
    This result holds under any partition of $A$ and $B$ into nonempty row blocks. If $A$ is just one block, then $Ae_W = 0$ has a solution if and only if $G$ has a uniformly weighted design of strength one.
\end{proof}

\section{Polyhedral Geometry of 2-valued Configurations}
\label{sec:2-valued}

In this section we develop the geometry of $2$-valued configurations, which will say more about \autoref{qn:exist} and help us with \autoref{qn:size}.  
If $\cP$ is $2$-valued with respect to $a^\top x$, then there are two faces $F_0$ and $F_1$ of the polytope $P = \conv(\cP)$ that minimize and maximize $a^\top x$ over $P$ respectively, and all points of $\cP$ lie in these two faces.  In particular, a necessary condition for $\cP$ to be $2$-valued is that no $p_j$ lies in the interior of $P$, although we can have points $p_j$ in the relative interior of $F_0$ and $F_1$. We say that the faces $F_0$ and $F_1$, as well as the direction $a$, are \textbf{$2$-witnesses} of $\cP$.  We remark that a related, but far stronger notion, is that of $2$-level configurations:  $\cP$ (with polytope $P=\conv(\cP)$) is \emph{$2$-level} or \emph{compressed} if all facets of $P$ are $2$-witnesses of $\cP$ (see for example, \cite{gouveia-parrilo-thomas2010}).  

\begin{example}
    Consider $\cP = \{(1,1,1), (1,0,0),(0,1,0),(0,0,1)\}$, whose convex hull is a tetrahedron inside the unit cube $[0,1]^3$. A tetrahedron is $2$-level and every facet is a $2$-witness of $\cP$. However, the tetrahedron has edges that are also $2$-witnesses of $\cP$. For example, the edge containing $p_1$ and $p_2$ is supported by the hyperplane $x_1=1$ while $p_3$ and $p_4$ lie on the hyperplane $x_1=0$. 
\end{example}

By \autoref{cor:minimal-unif-designs}, a minimal uniformly weighted design of strength $t$ corresponds to a facet of the eigenpolytope $P_t$. All such minimal designs meet the dimension bound, but \autoref{qn:size} is not restricted to minimal designs. What can be said about the sizes of uniformly weighted designs of strength $t$ when they do exist, but none are minimal because no facet of $P_t$ is a 2-witness of $\cB_t$?
To account for lower-dimensional faces of $P$ with large cardinality, which correspond to small designs, we make the following definition.

\begin{definition}
    A face of $P=\conv(\cP)$ that is a $2$-witness of $\cP$ and contains at least $\dim(P)$ points of $\mathcal{P}$ is said to \textbf{meet the dimension bound}. 
\end{definition}

In \autoref{ex:triangular-honeycomb} we saw a graph in which there are uniformly weighted designs of strength $t=1$, but all of them violate the dimension bound; these designs are carried by faces of $P_1$ that do not meet the dimension bound. We remark that such examples do not exist when $t$ becomes large enough, making \autoref{qn:size} interesting only when $t \leq \lfloor \frac{n}{2} \rfloor$.  

\begin{lemma} \label{lem:half-size-dim-bound}
    If $s_t = \sum_{i=0}^t \dim \Lambda_i > \lfloor \frac{n}{2} \rfloor$ and $G$ has uniformly weighted designs of strength $t$, then there is one that meets the dimension bound.
\end{lemma}
\begin{proof} 
    Suppose there is a uniformly weighted design $W$ of strength $t$ but it does not meet the dimension bound. Then, by hypothesis, $|W| > s_t \geq \lfloor \frac{n}{2} \rfloor + 1$, and so $|\overline{W}|  < \lfloor \frac{n}{2} \rfloor < s_t$. Recall that 
    $\overline{W}$ is again a uniformly weighted design of strength $t$, and it meets the dimension bound.
\end{proof}

\begin{example} \label{ex:hypercube-half-design}
The \textbf{hypercube graph} $H(n, 2)$ is the graph with vertex set $\FF_2^n$ and an edge between vertices that differ in one coordinate. Every $y \in \FF_2^n$ creates an eigenvector $\chi_y$ of the Laplacian of $H(n,2)$ with eigenvalue $2|y|$, given by $\chi_y(x) = (-1)^{y \cdot x}$.  Here $|y|$ denotes the Hamming weight of $y$. Therefore, $H(n,2)$ has $n$ nontrivial eigenspaces, with the $i$-th eigenspace of dimension $\binom{n}{i}$ and eigenvalue $2i$. 

\begin{figure}[!h]
    \begin{minipage}{0.48\textwidth}
        \centering
        \begin{tikzpicture}[scale=1]
            \tikzstyle{every node}=[circle, fill=white, draw=black, inner sep=2pt]
            \foreach \x in {0, 1} {\foreach \y in {0, 1} {\foreach \z in {0, 1} {\foreach \w in {0, 1} {
                \node (v\x\y\z\w) at (-0.2*\x + 1*\y + 1.4*\z + 1*\w, 1.4*\x + 1*\y + 0.2*\z - 1*\w) {};
            }}}}
            \foreach \x in {0, 1} {\foreach \y in {0, 1} {\foreach \z in {0, 1} {
                \draw (v\x\y\z0) -- (v\x\y\z1);
                \draw (v\x\y0\z) -- (v\x\y1\z);
                \draw (v\x0\y\z) -- (v\x1\y\z);
                \draw (v0\x\y\z) -- (v1\x\y\z);    
            }}}
        \end{tikzpicture}    
    \end{minipage}
    \begin{minipage}{0.48\textwidth}
        \centering
        \begin{tikzpicture}[scale=2]
            \tikzstyle{every node}=[circle, fill=white, draw=black, inner sep=2pt]
            \foreach \x in {0, 1} {
                \foreach \y in {0, 1} {
                    \foreach \z in {0, 1} {
                        \node (v\x\y\z) at (\x, \y, -\z) {};
                    }
                }
            }
            \foreach \xyz in {000, 110, 101, 011} {
                \node[fill=black] at (v\xyz) {};
            }
            \foreach \x in {0, 1} {
                \foreach \y in {0, 1} {
                    \draw (v\x\y0) -- (v\x\y1);
                    \draw (v\x0\y) -- (v\x1\y);
                    \draw (v0\x\y) -- (v1\x\y);
                }
            }
        \end{tikzpicture}
    \end{minipage}
    \caption{Left: the hypercube $H(4, 2)$. Right: the hypercube $H(3, 2)$ with a design averaging the first two nontrivial eigenspaces.}
    \label{fig:hypercube-designs}
\end{figure}
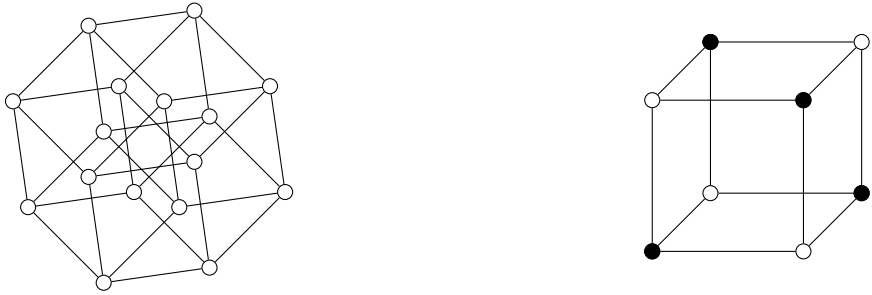
The last eigenspace of the hypercube $H(n, 2)$ is spanned by the single eigenvector $\chi_{\ones}$. It takes on value 1 at vertices with an even number of 0s, and value -1 otherwise. Therefore, for all nontrivial strengths $1 \le t \le n-1$, the last coordinate coming from $\chi_{\ones}$ is a 2-witness direction of the eigenpolytope. This direction produces two designs with $2^{n-1}$ vertices each. If $t \ge n/2$, then $s_t = \sum_{i=0}^t \dim \Lambda_i = \sum_{i=0}^t \binom{n}{i} \ge 2^{n-1}$, and these two designs meet the dimension bound as predicted by \autoref{lem:half-size-dim-bound}.
\end{example}

\begin{lemma} \label{lem:3dneedsfacet}
     If $\cP \subset \RR^d$ is $2$-valued, with $P = \conv(\cP)$ and $d = \dim(P) \le 3$, then $\cP$ always has a $2$-witness that is a facet of $P$. 
\end{lemma} 

\begin{proof}
    If $d=1$, then $P$ is a closed segment in $\RR$ and all points in $\cP$ lie at one of the two endpoints. Both endpoints are facets of $P$ that are $2$-witnesses of $\cP$.

    If $d=2$, then $\cP$ contains at least three distinct points. If a pair of faces $F_0$ and $F_1$ are $2$-witnesses of $\cP$, then one of them must contain at least two distinct points, which makes it a facet of $P$. 
        
    Finally, suppose $d=3$ and $F_0$ and $F_1$ are a pair of faces that are $2$-witnesses of $\cP$.  If neither is a facet of $P$, then they are vertices or edges of $P$. If either is a vertex, then $P$ is a segment or triangle, which contradicts the fact that $d=3$. Therefore $F_0$ and $F_1$ are edges and $P$ is a tetrahedron. A tetrahedron is $2$-level, and all of its facets are $2$-witnesses of $\cP$.
    \end{proof}

On the contrary, when $d \geq 4$, a 2-valued configuration $\cP \subset \RR^d$ may have no 2-witness facets in $P$.

\begin{proposition} \label{prop:polyhedra-with-no-2-witness-facet}
    For a fixed $d \ge 4$, define 
    \[
        u_1 = e_1 + e_d, \quad u_2 = e_1 + e_2 + \cdots + e_{d-1}, \quad u_3 = e_1 + e_2 + \cdots + e_d, 
    \]
    and consider $\cP_d = \{e_2, e_3, \dots, e_d\} \cup \{u_1, u_2, u_3\} \subset \RR^d$. Let $P_d = \conv(\cP_d)$. Then:
    \begin{enumerate}[(a)]
        \item $\cP_d$ is $2$-valued, with $a = e_1$ witnessing values $\{0,1\}$.
        \item $P_d$ is a simplicial polytope (all of its faces are simplices). 
        \item No facet of $P_d$ is a $2$-witness of 
        $\cP_d$, and no $2$-witness face of $P_d$ meets the dimension bound. 
    \end{enumerate}
\end{proposition} 
\begin{figure}[ht] 
    \centering
    \begin{tikzpicture}[scale=0.6] 
        \tikzstyle{every node}=[circle, fill=black, draw=black, inner sep=1pt]
        \tikzstyle{bgbox}=[fill=gray!20, draw=gray!40, line width=.55pt, rounded corners]
        
        \draw[bgbox] (-1, -2) -- (-1, 4) -- (7, 4) -- (7, -2) -- cycle;
        \draw[bgbox] (11.5, -2) -- (11.5, 4) -- (19, 4) -- (19, -2) -- cycle;
        \tikzstyle{labelnode}=[rectangle, draw=none, fill=none, inner sep=0pt]

        \node (e2) at (0, 0) [label=below left:$e_2$]{};
        \node (e3) at (4, -1) [label=below:$e_3$]{};
        \node (ed) at (6, 1) [label=$e_d$]{};
        \node (ei) at (2, 3) [label=above left:$\dots$]{};
        
        \draw (e2) -- (e3) -- (ed) -- (ei);
        \draw (e3) -- (ei) -- (e2) -- (ed);

        \node (u1) at (12.5, -1) [label=below right:$u_1$]{};
        \node (u2) at (18, 0) [label=$u_2$]{};
        \node (u3) at (15, 3) [label=above right:$u_3$]{};

        \draw (u1) -- (u2) -- (u3) -- (u1);
        
        \draw[dashed, gray] (u1) -- (e3);
        \draw[dashed, gray] (u2) -- (ed);
        \draw[dashed, gray] (u3) -- (ei);
        \draw[dashed, gray] (u3) -- (ed);
        \draw[dashed, gray] (u2) -- (e3);
        
        \node[anchor=south west, labelnode] at (-0.9,-1.9) {$x_1 = 0$};
        \node[anchor=south east, labelnode] at (18.9,-1.9) {$x_1 = 1$};
    \end{tikzpicture}
    \caption{Schematic diagram of the configuration $\cP_d$ from \autoref{prop:polyhedra-with-no-2-witness-facet}. The hyperplanes along direction $e_1$ are marked. On the hyperplanes lie two faces of $P_d$, a $(d-2)$-simplex (left) and a triangle (right). Both faces are 2-witnesses of $\cP_d$.}
    \label{fig:non-faceted-config}
\end{figure}
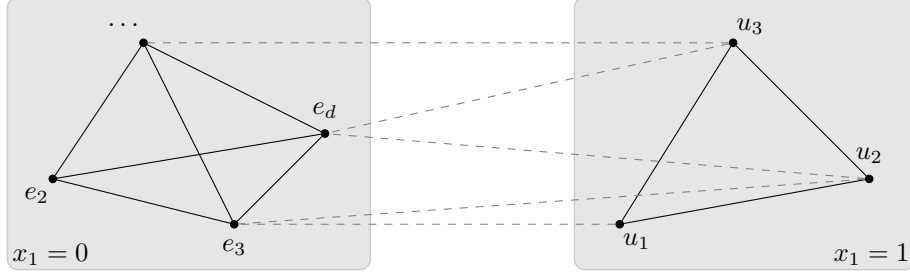
\begin{proof}
    Since all points in $\cP_d$ are in $\{0,1\}^d$, they are all vertices of the unit cube $[0,1]^d$. Therefore any coordinate direction $a = e_i$ is a $2$-witness of $\cP_d$. In particular $a = e_1$ is a $2$-witness, with one face the $(d-2)$-simplex formed by $\{e_2, e_3, \dots, e_d\}$, and the other face the triangle formed by $\{u_1, u_2, u_3\}$ (\autoref{fig:non-faceted-config}). Since neither are facets of $\cP_d$, we have (a).

    Since $P_d$ has $d+2$ vertices and dimension $d$, each facet contains at most $d+1$ vertices and at least $d$ vertices. To establish (b), we need to argue that no facet contains $d+1$ vertices. Suppose $P_d$ has a facet $F$ with $d+1$ vertices, and it maximizes the functional $c^\top x$. Up to an affine transform, we may assume that $c^\top x = 1$ on $F$. Then for each $i$ such that $e_i \in F$, we have $c_i = 1$. We have three cases, based on which vertex of $P_d$ is not on the facet $F$, and in each case we will get a contradiction:
    \begin{enumerate}
        \item If $e_j$ for $j=2, \dots, d$ are in $F$, then  $c_2 = \cdots = c_d = 1$, and $c_1$ is unknown. We compute that \[c^\top u_1 = c_1 + 1, \quad c^\top u_2 = c_1 + (d-2), \quad c^\top u_3 = c_1 + (d-1).\] These three values are different, but two of them must coincide as their corresponding vertices lie in $F$. A contradiction!
        \item If $e_j$ for some $j \ne d$ is not in $F$, then  $u_1, u_2, u_3 \in F$, but we check that $c^\top u_2 = c_1 + (d-3) + c_j \ne c_1 + (d-2) + c_j = c^\top u_3$. A contradiction!
        \item If $e_d \not \in F$, then both $u_1$ and $u_3$ are in $F$ in this case, but we have $c^\top u_1 = c_1 + c_d \ne c_1 + (d-2) + c_d = c^\top u_3$. A contradiction!
    \end{enumerate}

    To show part (c), one can do a similar check as above on all the possible subsets of size $d$ that might produce a facet. There is a similar splitting into cases based on whether the two excluded vertices are of the form $u_i$, $e_j$ with $j=2, \dots, d-1$ or $e_d$. One can show in all cases that the resulting facet direction has different values on the two excluded vertices, and thus no facet is a 2-witness. 

   Since $P_d$ is simplicial, any $2$-witness face with $d$ vertices would be a $(d-1)$-simplex. Such a face would be a facet of $P_d$, which we already eliminated. Hence, no $2$-witness face of $\cP_d$ meets the dimension bound.
\end{proof}
One of the central tools for constructions in this paper is the following polyhedral fact:

\begin{lemma} \label{lem:2-valued-projections}
Suppose $\cP \subset \RR^d$ is $2$-valued and there is a configuration $\cP' \subset \RR^{d'}$ (with $d' > d$) such that $\cP = \pi(\cP')$ for a linear map $\pi \,:\, \RR^{d'} \rightarrow \RR^d$. Then $\cP'$ is also $2$-valued. 
\end{lemma}
\begin{proof}
    Let $\pi^\perp$ be the projection of $\RR^{d'}$ into the orthogonal complement of $\RR^d$. For a $2$-witness direction $a$ of $\cP$, pick $a' \in \pi^{-1}(a)$ such that $\pi^\perp(a') = 0$. Then for any $p' \in \cP'$ we have 
    \[ a' \cdot p' = (\pi(a') + \pi^\perp(a')) \cdot (\pi(p') + \pi^\perp(p')) = a \cdot \pi(p') + a \cdot \pi^\perp(p') = a \cdot \pi(p').\]
    Since $\cP = \pi (\cP')$ is $2$-valued with respect to $a$, we get that $\cP'$ is also $2$-valued with $a'$ as 2-witness. 
\end{proof} 

We use the results above to showcase a graph in which there are uniformly weighted designs of some strengths, but not all strengths. 

\begin{example} \label{ex:icosahedron}
    Let $G$ be the icosahedral graph (the 1-skeleton of the icosahedron in $\RR^3$). It has four eigenspaces $\Lambda_0, \Lambda_1, \Lambda_2, \Lambda_3$, with Laplacian eigenvalues $0, (5-\sqrt{5})^{(3)}, 6^{(5)}, (5+\sqrt{5})^{(3)}$. Setting $\phi = \frac{1+\sqrt{5}}{2}$ (the golden ratio), and $\psi = \frac{1-\sqrt{5}}{2}$ (its conjugate), the matrix $U$ from \autoref{eq:U} is
    \[ U = \left( \begin{array}{rrrrrrrrrrrr}
            1 & 1 & 1 & 1 & 1 & 1 & 1 & 1 & 1 & 1 & 1 & 1 \\
            \hline
            \varphi & -\varphi & -\varphi & \varphi & -1 & -1 & 1 & 1 & 0 & 0 & 0 & 0 \\
            1 & -1 & -\varphi & \varphi & 0 & -\varphi & \varphi & 0 & 0 & 1 & -1 & 0 \\ 
            \varphi & -\varphi & -1 & 1 & 0 & -\varphi & \varphi & 0 & -1 & 0 & 0 & 1 \\
            \hline 
            -1 & -1 & 1 & 1 & 0 & 0 & 0 & 0 & 0 & 0 & 0 & 0 \\
            -1 & -1 & 0 & 0 & 1 & 1 & 0 & 0 & 0 & 0 & 0 & 0 \\
            -1 & -1 & 0 & 0 & 0 & 0 & 1 & 1 & 0 & 0 & 0 & 0 \\
            -1 & -1 & 0 & 0 & 0 & 0 & 0 & 0 & 1 & 1 & 0 & 0 \\
            -1 & -1 & 0 & 0 & 0 & 0 & 0 & 0 & 0 & 0 & 1 & 1 \\
            \hline
            \psi & -\psi & -\psi & \psi & -1 & -1    & 1    & 1 & 0  & 0 & 0  & 0 \\
            1    & -1    & -\psi & \psi & 0  & -\psi & \psi & 0 & 0  & 1 & -1 & 0 \\ 
            \psi & -\psi & -1    & 1    & 0  & -\psi & \psi & 0 & -1 & 0 & 0  & 1 
        \end{array} \right). \]

    For $t=2$, the eigenconfiguration $\cB_2$ consists of the $12$ (centrally symmetric) columns of the last three rows of $U$. Its convex hull $P_2$ is an (irregular) icosahedron in which all $12$ points of $\cB_2$ are vertices. Thus each facet contains three points, and each facet normal takes four distinct values on $\cB_2$. Since $P_2$ is $3$-dimensional, by \autoref{lem:3dneedsfacet}, $\cB_2$ is not 2-valued and $G$ has no uniformly weighted design of strength $2$.
    In \autoref{ex:alg-icosahedron} we will see a different reason for why there are no uniformly weighted designs in this case.

    When $t=1$, the eigenconfiguration $\cB_1$ consists of the $12$ columns from the last eight rows of $U$. Projecting this configuration onto its first five coordinates, the resulting points are the columns of the block of rows corresponding to $\Lambda_2$. The convex hull of this projection is a $5$-simplex with each point repeated twice at each of the six vertices. Since a simplex is $2$-level, by \autoref{lem:2-valued-projections}, $\cB_1$ is $2$-valued. Therefore, $G$ has uniformly weighted designs of strength $1$.
\end{example}

\begin{example} \label{ex:cube-smallsize-design}
    In \autoref{ex:hypercube-half-design} we saw that taking half of the vertices of the hypercube can produce a uniformly weighted design for all strengths. By applying \autoref{lem:2-valued-projections}, we can do even better.

     Let $t < \floor{2n/3}$. Pick $y_1, y_2, y_3 \in \FF_2^n$, with Hamming weights either $\floor{2n/3}$ or $\ceil{2n/3}$, so that in each coordinate two of the $y_i$ have a 1 and the other 0. Here $|y_i| > t$, so the corresponding eigenvector $\chi_{y_i}$ is a coordinate of the eigenconfiguration $\cB_t$. Moreover, $y_1 + y_2 + y_3 = 0$, and so for any $x$ we have
    \[
        \chi_{y_1}(x)\chi_{y_2}(x)\chi_{y_3}(x) = (-1)^{y_1 \cdot x}(-1)^{y_2 \cdot x}(-1)^{y_3 \cdot x} = (-1)^{(y_1 + y_2 + y_3) \cdot x} = (-1)^{0 \cdot x} = 1.
    \]
    Therefore the triple $(\chi_{y_1}(x), \chi_{y_2}(x), \chi_{y_3}(x))$ can only take on four out of eight possible values (the ones with product $1$). Considering the changes when we flip a bit in $x$, we can see that the four possibilities happen equally many times as $x$ ranges over $\FF_2^n$.

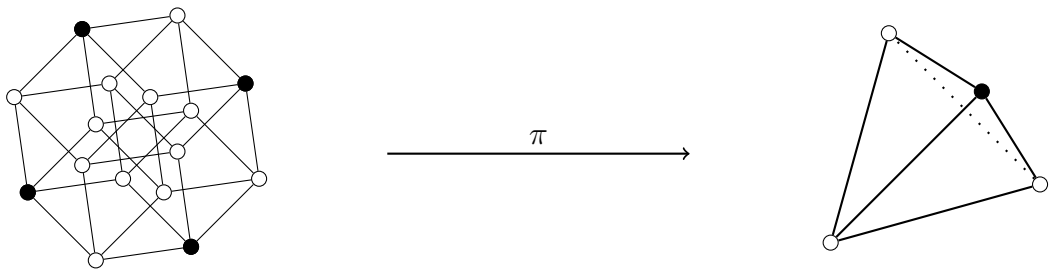
\begin{figure}[h!]
    \centering
    \begin{minipage}{0.36\textwidth}
        \centering
        \begin{tikzpicture}[scale=.9]
            \tikzstyle{every node}=[circle, fill=white, draw=black, inner sep=2pt]
            \foreach \x in {0, 1} {\foreach \y in {0, 1} {\foreach \z in {0, 1} {\foreach \w in {0, 1} {
                \node (v\x\y\z\w) at (-0.2*\x + 1*\y + 1.4*\z + 1*\w, 1.4*\x + 1*\y + 0.2*\z - 1*\w) {};
            }}}}
            \foreach \x in {0000, 0011, 1100, 1111} {
                \node[fill=black] at (v\x) {};
            }
            \foreach \x in {0, 1} {\foreach \y in {0, 1} {\foreach \z in {0, 1} {
                \draw (v\x\y\z0) -- (v\x\y\z1);
                \draw (v\x\y0\z) -- (v\x\y1\z);
                \draw (v\x0\y\z) -- (v\x1\y\z);
                \draw (v0\x\y\z) -- (v1\x\y\z);    
            }}}
        \end{tikzpicture}
    \end{minipage}
    \hfill
    \begin{minipage}{0.24\textwidth}
        \centering
        \begin{tikzpicture}
            \draw[->, thick] (0,0) -- (4,0) node[midway, above] {\Large $\pi$};
        \end{tikzpicture}
    \end{minipage}
    \hfill
    \begin{minipage}{0.36\textwidth}
        \centering
        \begin{tikzpicture}[
            scale=1,	
            back/.style={loosely dotted, thick},
            edge/.style={color=black, thick},
        ]
            \tikzstyle{every node}=[circle, fill=white, draw=black, inner sep=2pt]
            \node[fill=black] (A) at (1, 1, 1) {};
            \node (B) at (1, -1, -1) {};
            \node (C) at (-1, 1, -1) {};
            \node (D) at (-1, -1, 1) {};
    
            \draw[edge] (A) -- (B) -- (D) -- (C) -- (A);
            \draw[edge] (A) -- (D);
            \draw[edge, back] (B) -- (C);
        \end{tikzpicture}
    \end{minipage}
    \caption{A hypercube projecting to a tetrahedron, as described in \autoref{ex:cube-smallsize-design}. The marked vertices are preimages of a vertex of the tetrahedron, and form a design.}
    \label{fig:simplex-proj}
\end{figure}

    Let $\pi$ be the projection from the eigenpolytope $P_{t}$ onto the three coordinates corresponding to eigenvectors $\chi_{y_1}, \chi_{y_2}, \chi_{y_3}$. Then the above argument shows that $\pi(P_{t})$ is a tetrahedron (\autoref{fig:simplex-proj}), which has $2$-witness directions corresponding to a single vertex and the opposite triangle. By \autoref{lem:2-valued-projections}, the preimage of such a direction contains a $2$-witness direction on $P_{t}$, where one face contains a fourth and the other three fourths of all vertices. Therefore we have a design of size $2^{n-2}$ averaging the first $t=\floor{2n/3}$ eigenspaces.
\end{example}

We close out this section with two more polyhedral results. We have seen that lower-dimensional faces of $P$ can be $2$-witnesses of $\cP$. Such a face has infinitely many supporting hyperplanes, and so we might wonder which ones make the face a $2$-witness. It turns out that the 2-witness direction is unique for each face!

\begin{lemma} \label{lem:uniqueness-of-2-valued-direction}
    Any face in a full-dimensional polytope $P$ has at most one $2$-witness direction up to rescaling.
\end{lemma} 
\begin{proof}
    For this proof, we need to be a bit more careful about the definition of 2-witness direction, to avoid rescaling issues. Call $a$ a \emph{true 2-witness direction} if there exists a constant $a_0$ such that $a_0 + a^\top p = 1$ for all $p \in F$ and $a_0 + a^\top v = 0$ for all vertices $v$ of $P$ that do not lie on the face $F$. It is easy to check that any 2-witness direction can be rescaled to a true 2-witness direction. Suppose $v_1, v_2, \dots v_n$ are the vertices of $P$ and $e_F \in \{0, 1\}^n$ is the indicator vector of $F$. Then true 2-witness directions are precisely those vectors $a$ that have a corresponding $a_0$ such that $\begin{pmatrix}a_0 & a^\top\end{pmatrix}^\top$ is a solution to the equation    
    \begin{align} \label{eq:2-witness-solutions}
        \begin{bmatrix} 1  & v_1^\top \\ 1 & v_2^\top \\ \vdots & \vdots \\ 1 & v_n^\top  \end{bmatrix} \bx = e_F.
    \end{align}
    
    Suppose $a, b$ are distinct true $2$-witness directions for a face $F$. Both $\begin{pmatrix}a_0 & a^\top\end{pmatrix}^\top$ and $\begin{pmatrix}b_0 & b^\top\end{pmatrix}^\top$ are solutions to \autoref{eq:2-witness-solutions}, and therefore so is any affine combination $\begin{pmatrix}\mu a_0 + (1-\mu)b_0 & \mu a^\top + (1-\mu)b^\top\end{pmatrix}^\top$. We conclude that all affine combinations of $a$ and $b$ are true 2-witness directions of $F$.
    
    Now, if $a$ is a true 2-witness direction of $F$, we have $a^\top p = 1 - a_0$ for all $p \in F$, which is greater than the value $a^\top v = - a_0$ achieved at vertices of $P$ not on $F$. Therefore the functional $a^\top \bx$ is maximized over $P$ at $F$, and so $a$ lies on the normal cone of $F$. Therefore if we have distinct true $2$-witness directions $a, b$ of $F$, all affine combinations $\mu a + (1-\mu) b$ are in the normal cone of $F$. Consequently, the normal cone contains a nontrivial affine subspace which contradicts the fact that the polytope is full-dimensional.
\end{proof}

\begin{lemma} \label{lem:centrally-sym-2-face}
    Suppose $P = \conv(\cP)$ is centrally symmetric and full-dimensional, and all elements of $\cP$ are vertices of $P$. 
    Then any pair of $2$-witness faces of $P$  will each contain half the vertices.
\end{lemma}
\begin{proof} 
    Suppose $F$ and $\overline{F}$ are a pair of $2$-witness faces of $P$, and $a$ is the $2$-witness direction. Each vertex of $P$ lies on either $F$ or $\overline{F}$ but not both. Let $F'$ be the reflection of $F$ across the origin. Then $F'$ is also a face of $P$ as $P$ is centrally symmetric. It is disjoint from $F$ as the hyperplanes normal to $a$ containing $F$, $F'$ and the origin are all distinct. This then means that $F'$ is contained in $\overline{F}$.
    Now, we cannot have both a vertex $p$ and its mirror image $-p$ in the same face of a full-dimensional centrally symmetric polytope, since the face would have dimension equal to the polytope itself. Therefore if $p \in \overline{F}$, we have $-p \in F$ and thus $p \in F'$. This shows that $\overline{F}$ is contained in $F'$. We conclude that $\overline{F} = F'$.  
    Since $F$ and $F'$ have an equal number of vertices and all vertices lie on one of them, they each contain half the vertices of $P$.   
\end{proof}

\section{Graph Families with Uniformly Weighted Designs}
\label{sec:all-strengths}

In this section we describe several families of graphs with uniformly weighted designs of all strengths. Recall that for a graph $G$ with $k$ nontrivial Laplacian eigenspaces, we can construct eigenconfigurations $\cB_t$ and eigenpolytopes $P_t$ for $1 \le t \le k-1$. Uniformly weighted designs correspond to 2-witness directions in the eigenconfiguration. We provide two types of example families. In \autoref{sec:two-level-eigenpoly}, all eigenpolytopes in Laplacian order are $2$-level and hence every minimal positively weighted design of every strength is uniformly weighted. In \autoref{sec:threshold} and \autoref{sec:hypercube-simplex-proj}, we find families in which the eigenpolytopes are not $2$-level but are still $2$-valued, and hence there are uniformly weighted designs of every strength. 

\subsection{2-level Eigenpolytopes} \label{sec:two-level-eigenpoly}

Our first example is the family of {\bf cocktail party graphs}. These graphs are the 1-skeletons of cross-polytopes, have two nontrivial eigenspaces, and their minimal designs were described by Babecki and Thomas \cite[\S 4.1]{babecki2022galeduality}. In either ordering of the eigenspaces, the resulting eigenpolytopes are $2$-level, and hence all facet normals are $2$-witnesses and all minimal designs are uniformly weighted.

Another family with the same property consists of the  graphs ${\bf C_6 + \overline{K_n}}$ for $n\ge 6$; in \autoref{ex:C6+K4bar} we saw $C_6 + \overline{K_4}$. The Laplacians of all these graphs have six eigenspaces, and a similar analysis to \autoref{ex:C6+K4bar} shows that the eigenpolytopes $P_t$ for $1 \le t \le 4$ are $2$-level. Therefore, all minimal positively weighted designs of graphs in this family, of all strengths, are uniformly weighted. 

\medskip 

\begin{example} \label{ex:tripartite}
    \textbf{$K_{1,2,n}$: Complete Tripartite Graphs with tripartition $(1,2,n)$.}
    These graphs have $n+3$ vertices, and their Laplacians have four distinct eigenvalues $0,3^{(n-1)},(n+1)^{(1)}, (n+3)^{(2)}$ and a block structure. \autoref{fig:K_12n} shows the graph and its U-matrix (whose rows form a basis of Laplacian eigenvectors). 
    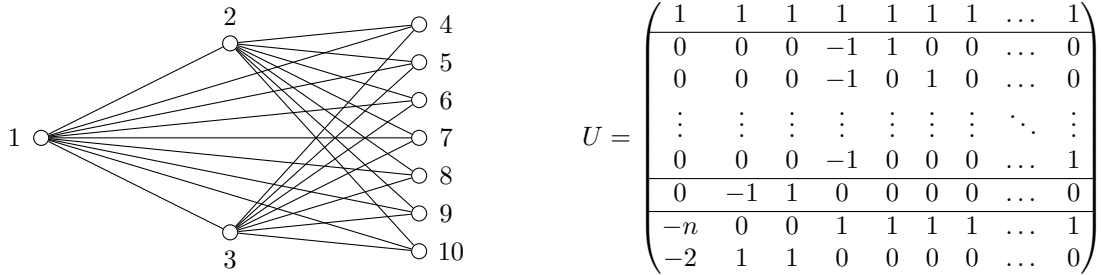
\begin{figure}[!h]
        \begin{minipage}{0.44\textwidth}
            \centering
            \begin{tikzpicture}[scale=0.5]
                \tikzstyle{every node}=[circle, fill=white, draw=black, inner sep=2pt]

                \node (a1) at (-5, 0) [label=left:1]{};
                \node (b2) at (0, 2.5) [label=2]{};
                \node (b3) at (0, -2.5) [label=below:3]{};
                \foreach \i[evaluate={\j=int(7 - \i)}] in {4,5,...,10}{
                    \node (c\i) at (5, \j) [label=right:\i]{};
                }

                \foreach \i in {4,5,...,10}{
                    \draw (c\i) -- (a1);
                    \draw (c\i) -- (b2);
                    \draw (c\i) -- (b3);
                }
                \draw (b2) -- (a1) -- (b3);
            \end{tikzpicture}
        \end{minipage}
        \begin{minipage}{0.52\textwidth}
            \centering
            $U = \begin{pmatrix}
                    1 & 1 & 1 & 1 & 1 & 1 & 1 & \dots & 1 \\
                    \hline
                    0 & 0 & 0 & -1 & 1 & 0 & 0 & \dots & 0 \\
                    0 & 0 & 0 & -1 & 0 & 1 & 0 & \dots & 0 \\
                    \vdots & \vdots & \vdots & \vdots & \vdots & \vdots & \vdots & \ddots & \vdots \\
                    0 & 0 & 0 & -1 & 0 & 0 & 0 & \dots & 1 \\
                    \hline
                    0 & -1 & 1 & 0 & 0 & 0 & 0 & \dots & 0 \\
                    \hline
                    -n & 0 & 0 & 1 & 1 & 1 & 1 & \dots & 1 \\
                    -2 & 1 & 1 & 0 & 0 & 0 & 0 & \dots & 0
                \end{pmatrix}$
        \end{minipage}
        \caption{Left: the graph $K_{1,2,n}$ for $n=7$. Right: The $U$-matrix of $K_{1,2,n}$ in general.}
        \label{fig:K_12n}
    \end{figure}

Since the Laplacian of $K_{1,2,n}$ has three nontrivial eigenspaces, the only strengths of interest are $1 \le t \le 3-1=2$. 
The eigenpolytope $P_2$ is a triangle with vertices $(-n,-2)$, $(0,1)$ and $(1,0)$ repeated with multiplicities $1,2,n$. This is a $2$-level polytope and carries uniformly weighted minimal designs of size $1,2,n$, supported exactly on the three parts of $K_{1,2,n}$, respectively.
The eigenpolytope $P_1$ is a tetrahedron with vertices $(0,-n,-2), (-1,0,1),(1,0,1),(0,1,0)$ with multiplicities $1,1,1,n$. This is also  $2$-level and carries minimal designs of sizes $1$ and $n$. 
Thus all minimal designs of these graphs are uniformly weighted. 
\end{example}

\subsection{Threshold Graphs} \label{sec:threshold}
The above three examples had $2$-level eigenpolytopes at all strengths. Such examples are usually on the simpler side, since their eigenpolytopes are heavily constrained. A much richer class is comprised of graphs whose eigenpolytopes are not necessarily $2$-level, but they are $2$-valued in each strength. \emph{Threshold graphs} are a big family of such graphs.

\begin{definition}
    A \textbf{threshold graph} has vertices $1, \dots, n$ such that vertex $i$ is either connected to all vertices before it (call these \textbf{dominating vertices}), or to none of the vertices before it (call these \textbf{isolated vertices}).
\end{definition}

Threshold graphs were introduced by Chv\'{a}tal and Hammer \cite{chvatal1977threshold}, as graphs with a linear threshold function separating independent sets from non-independent sets. They have many equivalent characterizations; see the book by Mahadev and Peled \cite{mahadevpeled1995threshold} for details.

\begin{example}
    We construct an explicit threshold graph $\cT(1,1,2,3,1)$ with 8 vertices in \autoref{fig:threshold-graph-picture}, which we will use as a running example throughout this section. In this example graph we add 2 as an isolated vertex, 3 and 4 as dominating vertices, 5, 6, and 7 as isolated vertices and finally 8 as a dominating vertex. 
\end{example}
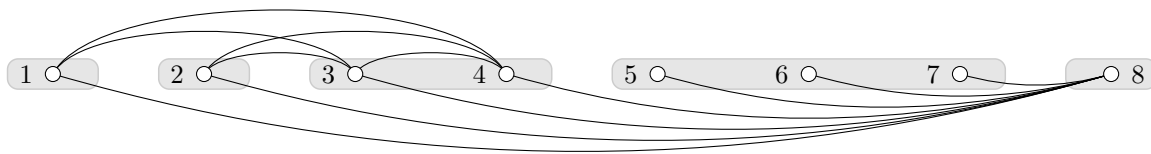
\begin{figure}[!h]
    \centering
    \begin{tikzpicture}[xscale=2]
        \tikzstyle{every node}=[circle, fill=white, draw=black, inner sep=2pt]

        \foreach \i in {1,...,8}{
            \coordinate (\i) at (\i, 0);
        }
        
        \tikzstyle{bgbox}=[fill=gray!20, draw=gray!40, line width=.55pt, rounded corners]

        \path[bgbox] (1) ++(0, 0.2) -- ++(0.3, 0) -- ++(0, -0.4) -- ++(-0.6, 0) -- ++(0, 0.4) -- ++(0.3, 0);
        \path[bgbox] (2) ++(0, 0.2) -- ++(0.3, 0) -- ++(0, -0.4) -- ++(-0.6, 0) -- ++(0, 0.4) -- ++(0.3, 0);
        \path[bgbox] (3) ++(0, 0.2) -- ++(1.3, 0) -- ++(0, -0.4) -- ++(-1.6, 0) -- ++(0, 0.4) -- ++(0.3, 0);
        \path[bgbox] (5) ++(0, 0.2) -- ++(2.3, 0) -- ++(0, -0.4) -- ++(-2.6, 0) -- ++(0, 0.4) -- ++(0.3, 0);
        \path[bgbox] (8) ++(0, 0.2) -- ++(0.3, 0) -- ++(0, -0.4) -- ++(-0.6, 0) -- ++(0, 0.4) -- ++(0.3, 0);

        \foreach \i in {1,...,7}{
            \draw (8) to[bend left=30] (\i);
        }
        \draw (3) to[bend right=75] (2);
        \draw (3) to[bend right=75] (1);
        \draw (4) to[bend right=75] (3);
        \draw (4) to[bend right=75] (2);
        \draw (4) to[bend right=75] (1);

        \foreach \i in {1,...,7}{
            \node (\i) at (\i) [label=left:\i]{};
        }
        \node (8) at (8) [label=right:8]{};
    \end{tikzpicture}
    \caption{Threshold graph $\cT(1,1,2,3,1)$. Boxes mark intervals of vertices of same type.}
    \label{fig:threshold-graph-picture}
\end{figure}

All connected threshold graphs on $n$ vertices share a basis of integral eigenvectors \cite{borg2026thresholdeigenvectors, macharete2024thresholdeigenbasis}. We will use this result to explicitly compute the graphical designs of threshold graphs, and show that the resulting eigenpolytopes are not $2$-level but the eigenconfigurations are $2$-valued at every strength.

For each $i=2, \dots, n$, define the vector $\mathbf{x}^i$ as
\[
    \mathbf{x}^i_\ell = \begin{cases}
        1       & \text{ if } \ell < i, \\
        - (i-1) & \text{ if } \ell = i, \\
        0       & \text{ if } \ell > i.
    \end{cases}
\]
We also define $\mathbf{x}^1$ to be the all-ones vector $\ones_n$. Note that the $\mathbf{x}^i$ are pairwise orthogonal. In fact, they form an orthogonal basis of eigenvectors for threshold graphs with $n$ vertices.
\begin{lemma}[Theorem 3.4 and Corollary 3.5 of \cite{borg2026thresholdeigenvectors}] \label{lem:threshold-eigenbasis}
    Suppose $G$ is a connected threshold graph, and let $\rho_i$ be the degree of vertex $i$. Then $\mathbf{x}^i$ is an eigenvector of the Laplacian of $G$. The eigenvalue of $\bx^i$ is $0$ if $i=1$, $\rho_i$ if vertex $i$ is isolated, and $\rho_i + 1$ if vertex $i$ is dominating.
\end{lemma}

If the vertices $i$ and $i+1$ are both dominating or both isolated, they have the same degree in the graph. In fact they are indistinguishable; there is a graph automorphism that swaps the two vertices and keeps everything else the same. 
Based on this, we can split $[n]$ into $k+1$ intervals $I_0, \dots, I_k$, such that each interval is comprised of vertices of the same type, consecutive intervals have vertices of different types, and vertex 1 is in an interval by itself (since it is trivially both isolated and dominating). Since all vertices in an interval $I_j$ share degree and type, the corresponding eigenvectors share an eigenvalue. In fact the subspaces $\Lambda_j = \textup{span}\paren{\set{\mathbf{x}^i: i \in I_j}}$ are precisely the eigenspaces of the Laplacian of the threshold graph.

\begin{example} \label{ex:threshold-U-matrix}
    For $n=8$, the vectors $\bx^i$ are written down on the left side of \autoref{eq:threshold-eigenvectors}. 

    \begin{align} \label{eq:threshold-eigenvectors}
        \begin{pmatrix}
            \bx^1 \\ \bx^2 \\ \bx^3 \\ \bx^4 \\ \bx^5 \\ \bx^6 \\ \bx^7 \\ \bx^8
        \end{pmatrix} = \begin{pmatrix} 
            1 & 1 & 1 & 1 & 1 & 1 & 1 & 1 \\
            \hline
            1 & -1 & 0 & 0 & 0 & 0 & 0 & 0 \\
            \hline
            1 & 1 & -2 & 0 & 0 & 0 & 0 & 0 \\
            1 & 1 & 1 & -3 & 0 & 0 & 0 & 0 \\
            \hline
            1 & 1 & 1 & 1 & -4 & 0 & 0 & 0 \\
            1 & 1 & 1 & 1 & 1 & -5 & 0 & 0 \\
            1 & 1 & 1 & 1 & 1 & 1 & -6 & 0 \\
            \hline
            1 & 1 & 1 & 1 & 1 & 1 & 1 & -7 \\
        \end{pmatrix}; \;
        U = \begin{pmatrix} 
            1 & 1 & 1 & 1 & 1 & 1 & 1 & 1 \\
            \hline
            1 & 1 & 1 & 1 & -4 & 0 & 0 & 0 \\
            1 & 1 & 1 & 1 & 1 & -5 & 0 & 0 \\
            1 & 1 & 1 & 1 & 1 & 1 & -6 & 0 \\
            \hline
            1 & -1 & 0 & 0 & 0 & 0 & 0 & 0 \\
            \hline
            1 & 1 & -2 & 0 & 0 & 0 & 0 & 0 \\
            1 & 1 & 1 & -3 & 0 & 0 & 0 & 0 \\
            \hline
            1 & 1 & 1 & 1 & 1 & 1 & 1 & -7 \\
        \end{pmatrix}.    
    \end{align}

    By \autoref{lem:threshold-eigenbasis}, the vectors $\bx^i$ are eigenvectors of $\cT(1,1,2,3,1)$ (\autoref{fig:threshold-graph-picture}). The degrees of the eight vertices are 3, 3, 4, 4, 1, 1, 1, 7 respectively. The eigenvalue of $\bx^1$ is 0 (as $i=1$). When $i$ is an isolated vertex (i.e., $i=2,5,6,7$), the eigenvalue equals the degree, so $\bx^2$ has eigenvalue 3 and $\bx^5, \bx^6, \bx^7$ form an eigenspace with eigenvalue 1. Finally, if $i$ is a dominating vertex the eigenvalue equals degree plus one. Therefore $\bx^3$ and $\bx^4$ both have eigenvalue 5 and $\bx^8$ has eigenvalue 8. Reordering the eigenvectors by Laplacian eigenvalue, we obtain the $U$-matrix of the graph $\cT(1,1,2,3,1)$ (\autoref{eq:threshold-eigenvectors}, right).
\end{example}

We explicitly describe the reordering of eigenvectors for threshold graphs in the lemma below. 

\begin{lemma} \label{lem:threshold-ordering}
    Suppose a connected threshold graph has vertex set $[n]$ split into intervals $I_0, \dots, I_k$ by vertex type. Without loss of generality, assume $k$ is even. Then the Laplacian ordering of eigenspaces is 
    \[
        \Lambda_k > \Lambda_{k-2} > \dots > \Lambda_2 > \Lambda_1 > \Lambda_3 > \dots > \Lambda_{k-1} > \Lambda_0.
    \]
\end{lemma}
\begin{proof}
    The last vertex $n$ (contained in interval $I_k$) for a threshold graph must be dominating, or else the graph would not be connected. Therefore the eigenspaces $\Lambda_k, \Lambda_{k-2}, \dots, \Lambda_2$ correspond to intervals $I_k, I_{k-2}, \ldots, I_2$ of dominating vertices, and their eigenvalues are $\rho_j + 1$ (where $\rho_j$ is the degree of vertices in $I_j$). If $I_j, I_{j'}$ are intervals of dominating vertices with $j > j'$, then there exist some intervals of isolated vertices between them. Vertices in $I_j$ are connected to these isolated vertices (since they were added later), while vertices in $I_{j'}$ are not. Moreover, this is the only difference in adjacencies between vertices in $I_j$ and $I_{j'}$. Therefore vertices in $I_j$ have a higher degree than those in $I_{j'}$, which means in Laplacian order $\Lambda_j > \Lambda_{j'}$. Ordering all the eigenspaces for dominating vertices, we have $\Lambda_k > \Lambda_{k-2} > \dots > \Lambda_2$.

    Since vertex $n$ is dominating, the eigenspaces $\Lambda_1, \Lambda_3, \dots, \Lambda_{k-1}$ correspond to intervals $I_1, I_3, \dots, I_{k-1}$ of isolated vertices, and their eigenvalues are $\rho_j$. If we have intervals $I_j$ and $I_{j'}$ of isolated vertices with $j > j'$, the first is not connected to the dominating vertices in intervals between $j$ and $j'$ while the second is. 
    By a similar argument to the one above for dominating vertices, we conclude that vertices in $I_j$ have lower degree than those in $I_{j'}$ and consequently $\Lambda_j < \Lambda_{j'}$ in Laplacian order. Ordering all the eigenspaces for isolated vertices, we have $\Lambda_1 > \Lambda_3 > \dots > \Lambda_{k-1}$.
    
    Note that a dominating vertex can never have lower degree than an isolated vertex. Moreover, the eigenvalue of the corresponding eigenvector is $\rho_j + 1$ for a dominating vertex but $\rho_j$ for an isolated vertex. Therefore all the eigenspaces corresponding to intervals of dominating vertices  will be higher in Laplacian order than eigenspaces corresponding to intervals of isolated vertices.
    Finally $\Lambda_0$ will come last as it has eigenvalue 0, which is below the degree of all vertices in a connected graph. Thus the ordering of the eigenspaces is $\Lambda_k > \Lambda_{k-2} > \dots > \Lambda_2 > \Lambda_1 > \Lambda_3 > \dots > \Lambda_{k-1} > \Lambda_0.$
\end{proof}

\autoref{lem:threshold-ordering} for odd $k$ has an identical proof, but the resulting Laplacian ordering is $\Lambda_k > \Lambda_{k-2} > \dots > \Lambda_1 > \Lambda_2 > \Lambda_4 > \dots > \Lambda_{k-1} > \Lambda_0.$ We shall now use the lemma to characterize the uniformly weighted designs of the threshold graph. We show uniformly weighted designs exist at every strength, but not every facet of the eigenpolytope produces designs of uniform weight.
\begin{theorem} \label{thm:threshold-singleton-design}
    The vertex $n$ of a connected threshold graph is a singleton uniformly weighted design for all strengths. In other words, the eigenconfigurations $\cB_t$ are 2-valued for all strengths.
\end{theorem}
\begin{proof}
    At all strengths $1 \le t \le k-1$, the eigenconfiguration $\cB_t$ has a coordinate coming from the last eigenvector $\mathbf{x}^n$ (since $n \in I_k$). This vector has exactly two values: 1 for all vertices $i<n$ and $-(n-1)$ at the vertex $n$. Therefore, it creates a $2$-witness direction for $\cB_t$. One $2$-witness face of this direction is just the vertex $b_n \in P_t$, and the other is the convex hull of all other vertices of $P_t$. This other face will contain all the other elements of $\cB_t$. Taking complements, this shows that the vertex $n$ of the graph is a uniformly weighted design by itself for all strengths.
\end{proof}
\begin{theorem} \label{thm:threshold-not-2-level}
    Let $G$ be a connected threshold graph, with vertex set $[n]$ split into intervals $I_0, \dots, I_k$ by vertex type, and $k>2$. For strengths $1 \le t \le k-2$, the eigenpolytopes $P_t$ are simplices, but the eigenconfigurations $\cB_t$ have points in the interior of certain faces. Therefore any uniformly weighted design of strength $t$ contains all or none of the points in these faces. In particular, there are facet normals of $P_t$ that are not $2$-witness directions, and thus minimal positively weighted designs that are not uniformly weighted.
\end{theorem} 
\begin{proof}
    Let $\beta$ be the dimension of $\cB_t$. The points in $\cB_t$ fall into three cases:
        
    \begin{enumerate}
        \item \textbf{$b_i = \ones_\beta$.}  This occurs for $i$ such that all the rows $\bx^\ell$ used to produce $\cB_t$ have $i < \ell$ (which implies $\bx^\ell_i = 1$). All points of this type coincide in the eigenconfiguration. 

        \item \textbf{$b_i = (0, \dots, 0, -(i-1), 1, \dots, 1)$.} This occurs when $\bx^i$ is one of the rows used to produce $\cB_t$. The $-(i-1)$ in the point occurs at the coordinate corresponding to $\bx^i$. There are $\beta$ such points, and they are all distinct in the eigenconfiguration.
        \item \textbf{$b_i = (0, \dots, 0, 1, \dots, 1).$} This occurs when $\bx^i$ is not one of the rows used to produce $\cB_t$, but $i$ lies in an interval $I_j$ with rows from both $I_{j-1}$ and $I_{j+1}$ used to produce $\cB_t$. The transition from 0 to 1 happens between the coordinates coming from $I_{j-1}$ and $I_{j+1}$. All points coming from an interval $I_j$ correspond to the same point in the eigenconfiguration. 
    \end{enumerate}

    Points in cases (1) and (2) exist for all strengths $t$. For example, $b_1$ is always in case (1), while $b_i$ for $i \in I_k$ is always in case (2). 
    Points in case (3) exist when $1 \le t \le k-2$. For these strengths, $\cB_t$ always has eigenvectors from $\Lambda_k$ and $\Lambda_{k-2}$ as coordinates, but never the eigenvectors from $\Lambda_{k-1}$ (since, by \autoref{lem:threshold-ordering}, they are respectively the first, second and last nontrivial eigenspaces in Laplacian order). Therefore, all vertices in $I_{k-1}$ will be in case (3) for these strengths. 

    We can check that the points in cases (1) and (2) are affinely independent. There are $\beta + 1$ such points, so they form the vertices of a simplex.
    No point in case (3) is affinely independent from these $\beta+1$ simplex vertices. We show that they lie in the relative interior of faces of the simplex, by proving an intermediate claim: a point in case (3) will be the centroid of all the points coming before it. 
    At the coordinates $\ell$ with $(b_i)_\ell = 1$, every point $i' < i$ in the eigenconfiguration $\cB_t$ also has $(b_{i'})_\ell = 1$. On the other hand, a coordinate $\ell$ with $(b_i)_\ell = 0$ has all nonzero coordinates of $\bx^\ell$ appear among the $(b_{i'})_\ell$ with $i' < i$. Since $\bx^\ell \cdot \ones = 0$, this implies that $(b_i)_\ell = 0$ is the average of the $(b_{i'})_\ell$ with $i' < i$. Therefore we have the equation
    \[
        b_i = \frac1{i-1}\sum_{i'<i} b_{i'},
    \]
    which proves our claim.     
    The convex hull of the points $b_{i'}$ with $i' < i$ is a face of the simplex $P_t = \conv(\cB_t)$, and $b_i$ is in the relative interior of this face, since being in case (3) implies that some of the $b_{i'}$ are distinct. We call such faces \emph{thorny}.

    Suppose $F$ is a thorny face of $P_t$, with $b$ in case (3) in its relative interior. Then $b$ is the centroid of the vertices of $F$. For the sake of contradiction, suppose there is a 2-witness direction $a$ for $\cB_t$ such that $a^\top p_1 \ne a^\top p_2$ for vertices $p_1, p_2$ of $F$. We must have $a^\top b$ strictly between the values $a^\top p_1 \ne a^\top p_2$. This implies that the function $a^\top \bx$ takes on at least three distinct values among points in $\cB_t$, which shows that it cannot be a 2-witness direction! Thus all 2-witnesses take on the same value over all points in a thorny face, and any uniformly weighted design contains all or none of these points. 

    The point $b_1$ is contained in all thorny faces, since $1 < i$ for all $b_i$ in case (3). However, $b_1 = \ones_\beta$ is a vertex of the simplex $P_t$. We have a facet of the simplex that is the convex hull of all the other vertices. This facet contains some vertices of thorny faces but not $b_1$. Therefore we have a facet normal that is not a 2-witness direction, which produces a minimal positively weighted design that is not uniformly weighted.
\end{proof}

\begin{example}
    Let us look at the eigenconfiguration $\cB_2$ for our example threshold graph $\cT(1,1,2,3,1)$. This graph has four nontrivial eigenspaces, and we use eigenvectors from the last two to construct $\cB_2$. This corresponds to the last three rows of the $U$-matrix from \autoref{ex:threshold-U-matrix}, which we label $B$ (\autoref{fig:threshold-eigenpolytope}, left). The points $b_i \in \cB_t$ are given by the $i$-th column of $B$.

    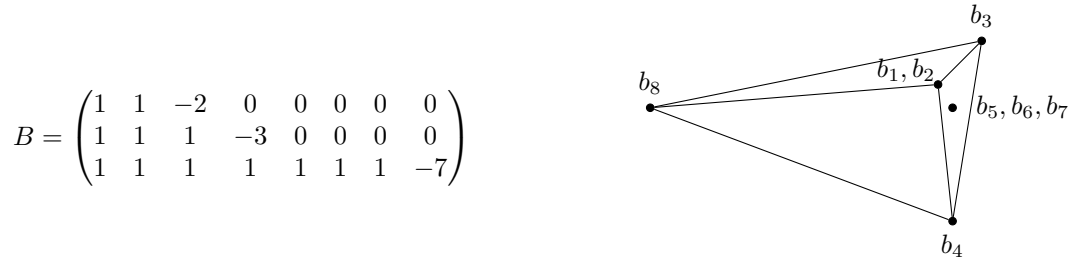
\begin{figure}[!h]
        \centering
        \begin{minipage}{0.48\textwidth}
            \centering
            \[B = \begin{pmatrix} 
                1 & 1 & -2 & 0 & 0 & 0 & 0 & 0 \\
                1 & 1 & 1 & -3 & 0 & 0 & 0 & 0 \\
                1 & 1 & 1 & 1 & 1 & 1 & 1 & -7 \\
            \end{pmatrix}\]
        \end{minipage}
        \begin{minipage}{0.48\textwidth}
            \centering
            \begin{tikzpicture}[scale=0.5]
                \tikzstyle{every node}=[circle, fill=black, draw=black, inner sep=1pt]
                \tikzstyle{labelnode}=[rectangle, draw=none, fill=none, inner sep=0pt]
    
                \node (12) at (1, 1, 1) {};
                \node (3) at (1, 1, -2) [label=$b_3$]{};
                \node (4) at (1, -3, 0) [label=below:$b_4$]{};
                \node (567) at (1, 0, 0) {};
                \node (8) at (-7, 0, 0) [label=$b_8$]{};
    
                \node[labelnode, above left=-0.05em of 12] {$b_1, b_2$}; 
                \node[labelnode, right=0.7em of 567] {$b_5, b_6, b_7$};
    
                \draw (12) -- (4) -- (3) -- (8);
                \draw (4) -- (8) -- (12) -- (3);
            \end{tikzpicture}
        \end{minipage}
        \caption{Eigenconfiguration $\cB_2$ for $\cT(1,1,2,3,1)$ is formed by columns of $B$ (left). Its convex hull is a tetrahedron, with points in the interior of a triangular face (right).}
        \label{fig:threshold-eigenpolytope}
    \end{figure}
    
    Drawing the eigenconfiguration (\autoref{fig:threshold-eigenpolytope}, right), we see that its convex hull is a tetrahedron with a vertex repeated twice ($b_1$, $b_2$). There is a point in the interior of a triangular face, repeated thrice ($b_5$, $b_6$, $b_7$). We can check that $b_8$ by itself is a 2-witness for this eigenconfiguration, as predicted by \autoref{thm:threshold-singleton-design}. 
    
    However, not all facet directions of the tetrahedron are 2-witnesses. For example, the facet formed by vertices $b_3, b_4$ and $b_8$ (the points in case (2) in the proof of \autoref{thm:threshold-not-2-level}) is given by the equation $16x_1 + 8x_2 + 3x_3 = -21$. The corresponding linear functional takes on the value 3 at $b_5, b_6, b_7$ (the points in case (3)) and the value 27 at $b_1, b_2$ (the points in case (1)). Therefore this facet normal direction takes on three values over $\cB_2$, and is not a 2-witness, as predicted by \autoref{thm:threshold-not-2-level}.
\end{example}

\subsection{Hypercube Graphs and Simplex Projections} \label{sec:hypercube-simplex-proj}

We have previously encountered the hypercube graph in \autoref{ex:hypercube-half-design} and \autoref{ex:cube-smallsize-design}. In this section we show that in all strengths, the hypercube graph has uniformly weighted designs of size {\em much smaller} than the dimension bound. While showing this, we provide a new geometric proof of a classical duality result. This opens up several new questions that are explored at the end of this subsection.

A \textbf{linear code} with \textbf{block length} $n$ and \textbf{minimal distance} $t+1$ is a subspace $S \subset \FF_2^n$ such that all nonzero elements $x \in S$ satisfy $|x| > t$. Then by linearity, for all pairs $x, y \in S$ we have $|x-y| > t$, or the minimal distance between points in the code is $t+1$. The dimension of the code is the dimension of the subspace; if $S$ is a linear code with dimension $d$, it contains $2^d$ elements. The \textbf{dual code} of a linear code is its orthogonal complement in $\FF_2^n$. The dual code of a dimension $d$ code has size $2^{n-d}$. It is a folklore result that the dual of a linear code with minimum distance $t+1$ is an orthogonal array with strength $t$ \cite{hedayat1999orthogonal}. 

It was shown in \cite{chowdhury2025combstructures} that graphical designs of the hypercube in Laplacian order correspond to binary orthogonal arrays. \autoref{ex:cube-smallsize-design} constructed such designs for strength $t \le \floor{2n/3}$ of size $2^{n-2}$, by considering projections to a tetrahedron. We generalize this construction by projecting to a higher-dimensional simplex, which provides a new geometric proof of the duality between linear codes and linear orthogonal arrays.

\begin{proposition} \label{prop:geom-proof-duality} 
    Suppose $S$ is a linear code of block length $n$, dimension $d$ and minimal distance $t+1$. The projection $\pi$ of $\cB_t$ onto the coordinates $y \in S, y \ne 0$ produces a $(2^d-1)$-simplex. This produces a strength $t$ design of the hypercube graph with $2^{n-d}$ points, which corresponds to the dual code of $S$.
\end{proposition}
\begin{proof}
    We first note that, since $|y| > t$ for all $0 \neq y \in S$, the eigenvectors $\chi_y$ for $0 \neq y \in S$ are not averaged by a design of strength $t$ in the hypercube graph. Therefore, it is possible to project $\cB_t$ onto all coordinates indexed by the nonzero elements in $S$. 

    Let $y_1, \dots, y_d$ be a basis of $S$. By linear independence, the $d$-tuples $(y_i \cdot x)_{i=1}^d$ take on all $2^d$ possible values as $x$ ranges over $\FF_2^n$. Moreover any $y \in S$ is a linear combination of the $y_i$, and so $\chi_y(x)$, which is based on the value of $y\cdot x$, is determined by the $d$-tuple $(y_i \cdot x)_{i=1}^d$. Therefore the $2^d$ possibilities for the $d$-tuple determine all possible images of the projection $\pi(\cB_t)$, which is a $(2^d-1)$-dimensional polytope with $2^d$ points i.e., a simplex.

    The map $x \mapsto (y \cdot x)_{y \in S, y \ne 0}$ is linear, and therefore all of its fibers have the same size. Therefore the projection $\pi: x \mapsto ((-1)^{y\cdot x})_{y \in S, y \ne 0}$ also has equally sized fibers, and $\pi$ is a $2^{n-d}$-to-1 cover from $\cB_t$ to the simplex $\pi(\cB_t)$.
    The simplex is a 2-level polytope, and hence all facet directions are 2-witnesses, corresponding to 2-witness faces containing 1 and $(2^d - 1)$ vertices. The preimages of these faces are faces of $P_t$ containing $2^{n-d}$ and $2^n - 2^{n-d}$ vertices. By \autoref{lem:2-valued-projections}, these faces are 2-witnesses of $\cB_t$. The complements of these faces are designs of $H(n, 2)$ with strength $t$ and size $2^{n-d}$ and $2^n - 2^{n-d}$ respectively. If we take the preimage of the facet not containing the origin, the corresponding design will be the dual code of $S$.
\end{proof}

The relationship between dual codes and designs was already described by Babecki \cite[Theorem 4.8]{babecki2021codes}, but not for Laplacian order. We shall apply this relationship for the specific case of the BCH code, to obtain uniformly weighted designs of the hypercube that beat the dimension bound in \autoref{ex:BCH-dual-design}. This is not novel; it has even been described (albeit in a different language) in textbooks \cite[Theorem 16.2.1]{alon2008probabilistic}.

\begin{example} \label{ex:BCH-dual-design}
    The BCH code is a particular type of linear code. For any $m \ge 3$ and $t < 2^{m-1}$, there exists a BCH code with block length $n = 2^m -1$, dimension at least $d = n - mt - 1$ and minimum distance at least $2t+1$. By \autoref{prop:geom-proof-duality}, the dual of the BCH code is a design of the graph $H(n, 2)$ averaging the first $2t$ eigenspaces, with size upper bounded by $2^{n-d} = 2^{mt+1} = O(n^t)$. 
    
    The number of eigenvectors of the $n$-dimensional hypercube averaged by a design of strength $2t$ is
    \[s_{2t} = \sum_{i=0}^{2t} \binom{n}{i} = O(n^{2t}).\]
    Therefore the dual of the BCH code has size roughly the square root of the dimension bound. Expanding this construction to all values of $n, d$ and $t$, we can show that the hypercube graph has uniformly weighted designs that meet the dimension bound at all strengths.
\end{example}

\autoref{prop:geom-proof-duality} and \autoref{ex:BCH-dual-design} are known results, but our work presents them from a new polyhedral perspective via the facial structure of eigenpolytopes. It would be interesting if this perspective can be applied beyond linear codes and linear orthogonal arrays. The following questions expand on this idea. 

\begin{question} 
    Every orthogonal array of strength $t$, being a uniformly weighted design of the hypercube graph, corresponds to a 2-witness direction in the eigenconfiguration $\cB_t$. For small orthogonal arrays, this means the eigenpolytope has 2-witness faces that contain a vast majority of the vertices. For example, a $4\ell \times 4\ell$ Hadamard matrix corresponds to a strength 2 orthogonal array in the $(4\ell - 1)$-cube. By the methods of this paper, the eigenconfiguration $\cB_2$ for $n = 4\ell - 1$ therefore has a face that contains $2^n - n - 1$ of the $2^n$ points, and moreover this face is a 2-witness. Could we employ polyhedral methods to find such faces, and produce novel constructions for Hadamard matrices or other nonlinear orthogonal arrays?
\end{question}

\begin{question}
    Just as the designs of the hypercube graph in Laplacian order are orthogonal arrays, the designs of the Johnson graph are combinatorial block designs. Could one find $2$-witness faces in the eigenpolytopes of the Johnson graph that meet the dimension bound? This would produce a new, geometric proof of the existence of combinatorial block designs \cite{keevash2024existencedesigns, kuperberg2017combstructures}.
\end{question}

\begin{question}
    Can the idea of ``duality'' be generalized beyond the construction in \autoref{prop:geom-proof-duality} coming from coordinate projections to a simplex? Is there a notion of duality for a nonlinear code, or for a combinatorial block design?
\end{question}

\section{Graph Families without Uniformly Weighted Designs} \label{sec:algebraic-methods-no-designs}

In this last section we construct families of graphs for which there are no uniformly weighted designs of any strength. 
We elaborate on an idea first presented in \cite[Lemma 6.6]{chowdhury2025combstructures}. 

\begin{theorem}[Rational Weights Theorem] \label{thm:rational-wts-thm}
    Suppose $\alpha \notin \QQ$ is an eigenvalue of the Laplacian of a graph $G$, with eigenspace $\Lambda$ of dimension $m$.
    Let $P(x)$ be the minimal polynomial of $\alpha$ over $\QQ$; the other roots of $P(x)$ are called \emph{conjugates} of $\alpha$. 
    If $\alpha'$ is a conjugate of $\alpha$, it will also be an eigenvalue of the Laplacian of $G$, with eigenspace $\Lambda'$ having dimension $m =\dim(\Lambda)$, and any design $W$ with rational weights averaging $\Lambda$ will also average $\Lambda'$.
\end{theorem}
\begin{proof}
    Let $K$ be the splitting field of $\alpha$; this is the smallest field containing all conjugates of $\alpha$. Given a conjugate $\alpha'$ of $\alpha$, we have a field automorphism $\sigma$ of $K / \QQ$ with $\sigma(\alpha) = \alpha'$. 

    Suppose $\phi_\alpha$ is an eigenvector of the Laplacian $L$ of $G$ with eigenvalue $\alpha$. This implies $(L - \alpha I)\phi_\alpha = 0.$ We apply the automorphism $\sigma$ to this equation. Since $L, I$ have rational entries they are fixed by $\sigma$. We have
    \[\paren{\sigma(L) - \sigma(\alpha)\sigma(I)} \sigma(\phi_\alpha) = (L - \alpha'I) \sigma(\phi_\alpha) = 0.\]
    Therefore, $\sigma(\phi_\alpha)$ is an eigenvector of $L$ with eigenvalue $\alpha'$. This shows that $\sigma$ is a bijection between the eigenspaces of $\alpha$ and $\alpha'$. Now, since $W$ with weights $a_w$ averages the eigenspace for $\alpha$,
    \[ \sum_{w \in W} a_w\phi_\alpha(w) =  \frac{1}{|V|} \sum_{v \in V} \phi_\alpha(v).\]        
    Applying $\sigma$ to the equation, we get
    \[ \sum_{w \in W} a_w\sigma(\phi_\alpha)(w) = \sum_{w \in W} \sigma(a_w)\sigma(\phi_\alpha(w)) =  \frac{1}{|V|} \sum_{v \in V} \sigma(\phi_\alpha)(v),\]     
    where $\sigma(a_w) = a_w$ because we assumed the design has rational weights. 
    Therefore $W$ with weights $a_w$ also averages $\sigma(\phi_\alpha)$, and thus averages the eigenspace $\Lambda'$ for $\alpha'$.
\end{proof}

\begin{corollary} \label{cor:conjugate-designs-unif-wted}
    If $G$ is a graph with eigenspaces $\Lambda$ and $\Lambda'$ whose corresponding eigenvalues $\alpha, \alpha'$ are conjugates,
    then any uniformly weighted design $W$ of $G$ either averages both $\Lambda$ and $\Lambda'$ or neither.
\end{corollary}

\begin{example} \label{ex:alg-icosahedron}
    The above results provide a much faster way to reach the same conclusion as in one of the parts of \autoref{ex:icosahedron}. The nontrivial eigenvalues of the icosahedral graph are $5 - \sqrt{5}, 6, 5 + \sqrt{5}$. The first and the third are conjugates; in fact, the field automorphism $\sigma: \phi \mapsto \psi$ of $\QQ(\phi)$ is a bijection between the first and third eigenspaces.
    By \autoref{cor:conjugate-designs-unif-wted} any uniformly weighted design will either average both $\Lambda_1$ and $\Lambda_3$ or neither. Therefore any uniformly weighted design of strength 2 must also average $\Lambda_3$, and thus be the whole vertex set. This shows that the eigenconfiguration $\cB_2$ cannot have any 2-witness directions. 
\end{example}

The above results suggest that factorizing the characteristic polynomial of the graph Laplacian is a key step in understanding the uniformly weighted designs of the graph. Now the Laplacian characteristic polynomial is never irreducible; there is always a factor of $x$, corresponding to the eigenvalue 0 with eigenvector $\ones$. Call a graph \textbf{Laplacian irreducible} if its Laplacian characteristic polynomial factors over the rationals into a single factor of $x$ and a degree $n-1$ irreducible factor (where $n$ is the number of vertices of the graph). 

\begin{proposition}
    A Laplacian irreducible graph has no nontrivial uniformly weighted designs at any strength.
\end{proposition}
\begin{proof}
    All nonzero eigenvalues of a Laplacian irreducible graph are conjugates of each other. Therefore, by \autoref{cor:conjugate-designs-unif-wted} any uniformly weighted design must average \emph{all} eigenspaces of the graph, and thus it cannot be a proper subset of the vertices \cite[Lemma 2.5]{babecki2022galeduality}.
\end{proof}

It is an interesting open question whether most graphs are generically Laplacian irreducible. The positive answer is strongly supported by experimental evidence, but not much has been proved. The answer to a simpler statement regarding the irreducibility of characteristic polynomials of random $\{\pm 1\}$ matrices \cite{ferberjainsahsawhney2023irreducibility} is conditional on the extended Riemann hypothesis. If a generic random graph was Laplacian irreducible, we could conclude that almost all graphs do not have uniformly weighted designs.

We do not tackle the question for random graphs, but instead provide a method for constructing infinite families of Laplacian irreducible graphs using Eisenstein's criterion.
Our methods are inspired by previous results \cite{yu2021irredcharpoly}, which dealt with the adjacency matrix rather than the Laplacian. 

We first establish some structural lemmas. For any graph $G$ with vertex $u \in V(G)$ and Laplacian $L(G)$, let $p_G$ denote the characteristic polynomial of $L(G)$, and $p_{G, u}$ the characteristic polynomial of the principal minor of $L(G)$ obtained by deleting the row and column corresponding to vertex $u$. Also, we say a polynomial $p$ is \textbf{Laplacian Eisenstein} if it is monic, has no constant term (or is divisible by $x$), has all nonleading coefficients even, and has a coefficient of $x$ that is not divisible by 4.  

\begin{lemma}\cite[Lemma 8]{guo2005eigenvalue} \label{lem:bridge-laplacian-calculation}
    Let $G, H$ be graphs with vertices $u \in G, v \in H$. Suppose $G_{uv}H$ is the bridge graph obtained by adding an edge between vertex $u$ in $G$ and vertex $v$ in $H$. Then the Laplacian characteristic polynomial of the bridge graph can be computed by the formula
    \[
        p_{G_{uv}H} = p_Gp_H - p_{G, u}p_H - p_Gp_{H, v}.
    \]
\end{lemma}
\begin{lemma} \label{lem:laplacian-eisenstein} 
    Suppose $G$ is a graph with $p_G$ Laplacian Eisenstein. Then $G$ is Laplacian irreducible.
\end{lemma}
\begin{proof}
    Since $p_G$ has no constant term, $p_G/x$ is also a polynomial. In $p_G/x$, the prime 2 divides all nonleading coefficients, does not divide the leading coefficient ($p_G/x$ is monic), and $2^2=4$ does not divide the constant term. By Eisenstein's criterion, $p_G/x$ is irreducible. Equivalently, $G$ is Laplacian irreducible.
\end{proof}
\begin{theorem} \label{thm:no-unif-design-inf-families}
    Suppose we have an ``\emph{anchor graph}'' $G$ of order $n$ and ``\emph{repeat graph}'' $H$ of order $m$, with vertices $u \in G$ and $v \in H$, satisfying the following conditions:
    \begin{enumerate}
        \item $p_H$ is Laplacian Eisenstein.
        \item The polynomial $p_{H, v}$ has odd constant term.
        \item $p_G$ is Laplacian Eisenstein.
        \item $n > m$ and $p_{G, u} \equiv p_{H, v}x^{n-m} \pmod 2$.
    \end{enumerate}
    Then we can construct an infinite family of graphs $G_0 = G, G_1, G_2, \dots$, with $G_{i+1}$ obtained by adding a bridge between $G_i$ and a copy of $H$, such that all the graphs $G_i$ are Laplacian irreducible.
\end{theorem}
\begin{proof}
    We obtain the infinite family by the following inductive construction. First relabel $G_0 := G$ and $u_0 := u$. For $i=0, 1, \dots$, the graph $G_{i+1}$ is obtained from $G_i$ by taking the disjoint union of $G_i$ and $H$, adding an edge between $u_i$ and $v$, and labeling the copy of $v$ in $G_{i+1}$ as $u_{i+1}$.

    We now prove by induction that $p_{G_i}$ is Laplacian Eisenstein for all $i$. The base case $i=0$ is condition (3). For the other base case $i=1$, we know $p_{G_1} = p_Gp_H - p_{G, u}p_H - p_Gp_{H, v}.$ By conditions (1) and (3), we know that $p_H \equiv x^m \pmod 2$ and $p_G \equiv x^n \pmod 2$ respectively. Combining with condition (4), we have
    \[p_{G_1} \equiv p_Gp_H - p_{G, u}p_H - p_Gp_{H, v} \equiv x^{n+m} - p_{H, v}x^{n} - p_{H, v}x^n \equiv x^{n+m} \pmod 2.\]
    
    Now, let $p[x^k]$ denote the coefficient of the monomial $x^k$ in the polynomial $p$. Since $p_G, p_H$ are Laplacian characteristic polynomials and thus divisible by $x$, we have $p_G[1] = p_H[1] = 0$. Using $n > m$ in condition (4) also tells us $p_{G, u}[1] = 0$. Then we can calculate    
    \begin{equation} \label{eq:Laplacian-charpoly-xterm}
    \begin{aligned}
        p_{G_1}[x] &= (p_Gp_H - p_{G, u}p_H - p_Gp_{H, v})[x]\\
        &= p_G[x]p_H[1] + p_G[1]p_H[x] - p_{G, u}[x]p_H[1] - p_{G, u}[1]p_H[x] - p_G[x]p_{H, v}[1] - p_G[1]p_{H, v}[x]\\
        &= -p_G[x]p_{H, v}[1].
    \end{aligned}
    \end{equation}
    We have $p_G[x] \equiv 2 \pmod 4$ by condition (3) and $p_{H, v}[1]$ is odd by condition (2). Therefore their product is not divisible by 4, and neither is $p_{G_1}[x]$. Also, $p_{G_1}$ is monic and divisible by $x$ as it is a Laplacian characteristic polynomial. So we can conclude that $p_{G_1}$ is Laplacian Eisenstein. 

    We are now ready to perform the induction step. 
    Suppose $p_{G_i}$ and $p_{G_{i-1}}$ are Laplacian Eisenstein. Since $u_i$ in $G_i$ is part of a bridge joining a copy of $H$ to a copy of $G_{i-1}$, deleting $u_i$ would disconnect $H$ from all other vertices in $G_{i-1}$. By the block structure of the Laplacian, we have $p_{G_i, u_i} = p_{G_{i-1}}p_{H, v}$. Therefore,
    \[p_{G_{i+1}} = p_{G_i}p_H - p_{G_i, u_i}p_H - p_{G_i}p_{H, v} = p_{G_i}p_H - p_{G_{i-1}}p_{H, v}p_H - p_{G_i}p_{H, v}.\]
    We know $p_{G_i} \equiv x^{n + im} \pmod 2$ and $p_{G_{i-1}} \equiv x^{n + (i-1)m} \pmod 2$, as they are Laplacian Eisenstein. Using this, we reduce the equation for $p_{G_{i+1}}$ to
    \[p_{G_{i+1}} \equiv x^{n+im+m} - x^{n+(i-1)m}x^{m}p_{H, v} - x^{n+im}p_{H, v} \equiv x^{n+(i+1)m} \pmod 2.\]

    A similar calculation to \autoref{eq:Laplacian-charpoly-xterm} shows that $p_{G_{i+1}}[x] = -p_{G_i}[x]p_{H, v}[1]$, as the other terms involve polynomials divisible by $x$. By the inductive hypothesis and condition (2), we have $p_{G_{i+1}}[x] \equiv -2 \ne 0 \pmod 4$. Moreover, $p_{G_{i+1}}$ is monic and has no constant term as it is a Laplacian characteristic polynomial. Therefore, $p_{G_{i+1}}$ is Laplacian Eisenstein. By induction, we conclude that all the $p_{G_i}$ are Laplacian Eisenstein, and by \autoref{lem:laplacian-eisenstein} this implies that all the $G_i$ in our infinite family are Laplacian irreducible.
\end{proof}

\begin{figure}[!h]
    \centering
    \begin{minipage}{0.32\textwidth}
        \centering
        \begin{tikzpicture}[scale=0.5]
            \tikzstyle{every node}=[circle, fill=white, draw=black, inner sep=2pt]
    
            \node (g0) at (0, 0) [label=below:$u_1$]{};
            \node (g1) at (1, 5) [label=below right:$u_2$]{};
            \node (g2) at (-2, 5) {};
            \node (g3) at (-1, 3) {};
            \node (g4) at (-4, 3) {};
            \node (g5) at (-3, 0) {};
            \node (g6) at (-5, 1) {};
    
            \draw (g0) -- (g1) -- (g2);
            \draw (g0) -- (g3) -- (g4);
            \draw (g0) -- (g5) -- (g6);
            \draw (g1) -- (g3);
            \draw (g2) -- (g4) -- (g6);
            \draw (g4) -- (g5);
        \end{tikzpicture}
    \end{minipage}
    \begin{minipage}{0.32\textwidth}
        \centering
        \begin{tikzpicture}[scale=0.5]
            \tikzstyle{every node}=[circle, fill=white, draw=black, inner sep=2pt]
            \node (h1) at (2, 0) [label=below:$v$]{};
            \node (v1) at (2.8, 4) {};
            \draw (h1) -- (v1);
        \end{tikzpicture}
    \end{minipage}
    \begin{minipage}{0.32\textwidth}
        \centering
        \begin{tikzpicture}[scale=0.5]
            \tikzstyle{every node}=[circle, fill=white, draw=black, inner sep=2pt]
            \node (h10) at (2.5, 5) [label=above:$v$]{};
            \node (h11) at (2.5, 3.5) {};
            \node (h12) at (4, 3.5) {};
            \node (h13) at (4, 2) {};
            \node (h14) at (2.5, 2) {};
            \node (h15) at (1.5, 2.75) {};

            \draw (h10) -- (h11) -- (h12) -- (h13) -- (h14);
            \draw (h11) -- (h14) -- (h15) -- (h11);
        \end{tikzpicture}
    \end{minipage}
    \caption{Left: the anchor graph $G$ with two choices of $u$. Middle: repeat graph $H_1$ for the monorail family. Right: repeat graph $H_2$ for the houseboat family.}
    \label{fig:no-unif-design-families-pieces}
\end{figure}
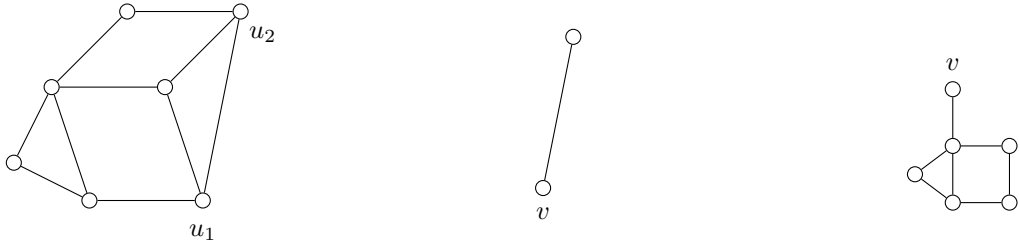

\begin{example} \label{ex:monorail-graphs}
    We define the infinite family of \emph{monorail graphs}, where our repeat graph $H_1$ is a single edge. In this case $p_{H_1} = x^2 - 2x$ and $p_{H_1, v} = x - 1$, which satisfies conditions (1) and (2) of \autoref{thm:no-unif-design-inf-families}. 
    There are many choices of anchor graph $G$; we use the graph in \autoref{fig:no-unif-design-families-pieces}, left. For this graph $G$,
    \[p_G = x^7 - 20x^6 + 160x^5 - 652x^4 + 1418x^3 - 1548x^2 + 658x \equiv x^7 \pmod 2.\]
    Also, the coefficient of $x$ in $p_G$ is $658$, which is not divisible by 4. Therefore $p_G$ is Laplacian Eisenstein, satisfying condition (3). Excluding the vertex $u_1$, the characteristic polynomial of the minor is \[p_{G, u_1} = x^6 - 17x^5 + 112x^4 - 360x^3 + 578x^2 - 416x + 94 \equiv x^6 - x^5 = x^5(x-1) \pmod 2,\] which satisfies condition (4). Therefore, following \autoref{thm:no-unif-design-inf-families} we can create an infinite family of graphs $G_i$ by adding copies of $H_1$ to $G$ by bridges. The monorail graphs are all Laplacian irreducible, and thus have no uniformly weighted designs at any strength.
\end{example}

\begin{figure}[!h]
    \centering
    \begin{minipage}{0.38\textwidth}
        \centering
        \begin{tikzpicture}[scale=0.5]
            \tikzstyle{every node}=[circle, fill=white, draw=black, inner sep=2pt]
            \tikzstyle{labelnode}=[rectangle, draw=none, fill=none, inner sep=0pt]

            \node (g0) at (0, 0) {};
            \node[labelnode, below=0.4em of g0] {$u$};
            \node (g1) at (1, 5) {};
            \node (g2) at (-2, 5) {};
            \node (g3) at (-1, 3) {};
            \node (g4) at (-4, 3) {};
            \node (g5) at (-3, 0) {};
            \node (g6) at (-5, 1) {};
    
            \draw (g0) -- (g1) -- (g2);
            \draw (g0) -- (g3) -- (g4);
            \draw (g0) -- (g5) -- (g6);
            \draw (g1) -- (g3);
            \draw (g2) -- (g4) -- (g6);
            \draw (g4) -- (g5);

            \node[fill=black] (h1) at (2, 0) {};
            \node[labelnode, below=0.4em of h1] {$v$};
            \node (v1) at (2.8, 4) {};
            \draw (g0) -- (h1) -- (v1);
        \end{tikzpicture}
    \end{minipage}
    \begin{minipage}{0.6\textwidth}
        \centering
        \begin{tikzpicture}[scale=0.5]
            \tikzstyle{every node}=[circle, fill=white, draw=black, inner sep=2pt]
            \tikzstyle{labelnode}=[rectangle, draw=none, fill=none, inner sep=0pt]

            \node (g0) at (0, 0) {};
            \node[labelnode, below=0.4em of g0] {$u$};
            \node (g1) at (1, 5) {};
            \node (g2) at (-2, 5) {};
            \node (g3) at (-1, 3) {};
            \node (g4) at (-4, 3) {};
            \node (g5) at (-3, 0) {};
            \node (g6) at (-5, 1) {};
    
            \draw (g0) -- (g1) -- (g2);
            \draw (g0) -- (g3) -- (g4);
            \draw (g0) -- (g5) -- (g6);
            \draw (g1) -- (g3);
            \draw (g2) -- (g4) -- (g6);
            \draw (g4) -- (g5);
    
            \node (h1) at (2, 0) {};
            \node (v1) at (2.8, 4) {};
            \node (h2) at (4, 0) {};
            \node (v2) at (4.8, 4) {};
            \node[draw=white] (dots) at (6, 0) {$\dots$};
            \node (hi) at (8, 0) {};
            \node (vi) at (8.8, 4) {};

            \node[labelnode, below=0.4em of h1] {$u_1$};
            \node[labelnode, below=0.4em of h2] {$u_2$};
            \node[labelnode, below=0.4em of hi] {$u_{i+1}$};

            \draw (g0) -- (h1) -- (v1);
            \draw (h1) -- (h2) -- (v2);
            \draw (h2) -- (dots) -- (hi) -- (vi);
        \end{tikzpicture}
    \end{minipage}
    \caption{The monorail graphs; initial bridging marked in black.}
    \label{fig:monorail-graph}
\end{figure}
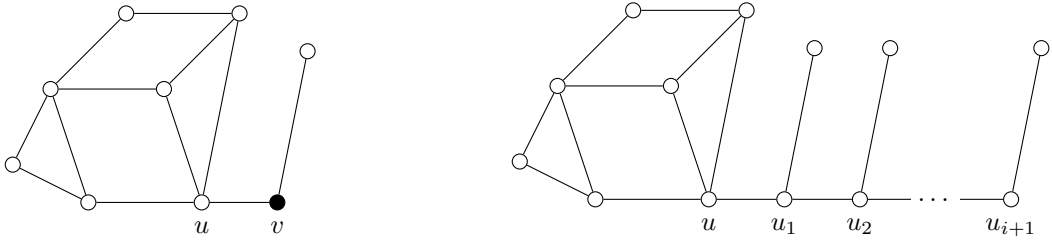

\medskip 

\begin{example} \label{ex:houseboat-graphs}
    We now define the family of \emph{houseboat graphs}. In this case our repeat graph is $H_2$ (\autoref{fig:no-unif-design-families-pieces}, right). The resulting graphs look like boats (or houseboats) lined up along a pier, hence the name. We have
    \begin{align*} 
    p_{H_2} & =  x^6 - 14x^5 + 72x^4 - 168x^3 + 176x^2 - 66x,\\
    p_{H_2, v} & =  x^5 - 13x^4 + 60x^3 - 117x^2 + 86x - 11 \equiv x^5 + x^4 + x^2 + 1 \pmod 2.
    \end{align*}
    and we can check that these satisfy conditions (1) and (2). We use the same anchor graph $G$ as the previous example (thus satisfying condition (3)), but choose a different vertex $u_2$ to bridge to. We calculate
    \[
        p_{G, u_2} = x^6 - 17x^5 + 112x^4 - 361x^3 + 586x^2 - 431x + 94 \equiv (x^5 + x^4 + x^2 + 1)x = p_{H_2, v}x \pmod 2 .
    \]
    Therefore $G$ and $H_2$ satisfy condition (4), and we can construct an infinite family by bridging copies of $H_2$ to the anchor. The resulting houseboat graphs are all Laplacian irreducible, and thus have no uniformly weighted designs at any strength.
\end{example}

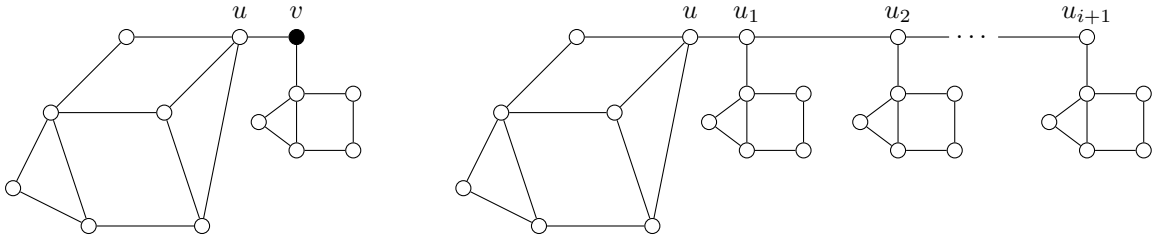
\begin{figure}[!h]
    \centering
    \begin{minipage}{0.38\textwidth}
        \centering
        \begin{tikzpicture}[scale=0.5]
            \tikzstyle{every node}=[circle, fill=white, draw=black, inner sep=2pt]
            \tikzstyle{labelnode}=[rectangle, draw=none, fill=none, inner sep=0pt]

            \node (g0) at (0, 0) {};
            \node (g1) at (1, 5) {};
            \node (g2) at (-2, 5) {};
            \node (g3) at (-1, 3) {};
            \node (g4) at (-4, 3) {};
            \node (g5) at (-3, 0) {};
            \node (g6) at (-5, 1) {};
            \node[labelnode, above=0.85em of g1, anchor=north] {$u$};

            \draw (g0) -- (g1) -- (g2);
            \draw (g0) -- (g3) -- (g4);
            \draw (g0) -- (g5) -- (g6);
            \draw (g1) -- (g3);
            \draw (g2) -- (g4) -- (g6);
            \draw (g4) -- (g5);

            \node[fill=black] (h10) at (2.5, 5) {};
            \node (h11) at (2.5, 3.5) {};
            \node (h12) at (4, 3.5) {};
            \node (h13) at (4, 2) {};
            \node (h14) at (2.5, 2) {};
            \node (h15) at (1.5, 2.75) {};
            \node[labelnode, above=0.85em of h10, anchor=north] {$v$};

            \draw (g1) -- (h10) -- (h11) -- (h12) -- (h13) -- (h14);
            \draw (h11) -- (h14) -- (h15) -- (h11);
        \end{tikzpicture}
    \end{minipage}
    \begin{minipage}{0.6\textwidth}
        \centering
        \begin{tikzpicture}[scale=0.5]
            \tikzstyle{every node}=[circle, fill=white, draw=black, inner sep=2pt]
            \tikzstyle{labelnode}=[rectangle, draw=none, fill=none, inner sep=0pt]

            \node (g0) at (0, 0) {};
            \node (g1) at (1, 5) {};
            \node (g2) at (-2, 5) {};
            \node (g3) at (-1, 3) {};
            \node (g4) at (-4, 3) {};
            \node (g5) at (-3, 0) {};
            \node (g6) at (-5, 1) {};
            \node[labelnode, above=0.85em of g1, anchor=north] {$u$};

            \draw (g0) -- (g1) -- (g2);
            \draw (g0) -- (g3) -- (g4);
            \draw (g0) -- (g5) -- (g6);
            \draw (g1) -- (g3);
            \draw (g2) -- (g4) -- (g6);
            \draw (g4) -- (g5);

            \node (h10) at (2.5, 5) {};
            \node (h11) at (2.5, 3.5) {};
            \node (h12) at (4, 3.5) {};
            \node (h13) at (4, 2) {};
            \node (h14) at (2.5, 2) {};
            \node (h15) at (1.5, 2.75) {};
            \node[labelnode, above=0.85em of h10, anchor=north] {$u_1$};

            \draw (g1) -- (h10) -- (h11) -- (h12) -- (h13) -- (h14);
            \draw (h11) -- (h14) -- (h15) -- (h11);

            \node (h20) at (6.5, 5) {};
            \node (h21) at (6.5, 3.5) {};
            \node (h22) at (8, 3.5) {};
            \node (h23) at (8, 2) {};
            \node (h24) at (6.5, 2) {};
            \node (h25) at (5.5, 2.75) {};
            \node[labelnode, above=0.85em of h20, anchor=north] {$u_2$};

            \draw (h10) -- (h20) -- (h21) -- (h22) -- (h23) -- (h24);
            \draw (h21) -- (h24) -- (h25) -- (h21);

            \node[draw=white] (dots) at (8.5, 5) {$\dots$};

            \node (hi0) at (11.5, 5) {};
            \node (hi1) at (11.5, 3.5) {};
            \node (hi2) at (13, 3.5) {};
            \node (hi3) at (13, 2) {};
            \node (hi4) at (11.5, 2) {};
            \node (hi5) at (10.5, 2.75) {};
            \node[labelnode, above=0.85em of hi0, anchor=north] {$u_{i+1}$};

            \draw (h20) -- (dots) -- (hi0) -- (hi1) -- (hi2) -- (hi3) -- (hi4);
            \draw (hi1) -- (hi4) -- (hi5) -- (hi1);

        \end{tikzpicture}
    \end{minipage}
    \caption{The houseboat graphs; initial bridging marked in black.}
    \label{fig:houseboat-graphs}
\end{figure} 

\printbibliography

@article{gouveia-parrilo-thomas2010,
  author  = {Gouveia, Jo\~{a}o and Parrilo, Pablo A. and Thomas, Rekha R.},
  title   = {Theta Bodies for Polynomial Ideals},
  journal = {SIAM Journal on Optimization},
  volume  = {20},
  number  = {4},
  pages   = {2097--2118},
  year    = {2010},
  doi     = {10.1137/090747587}
}

@article {babecki-shiroma,
    AUTHOR = {Babecki, Catherine and Shiroma, David},
     TITLE = {Eigenpolytope universality and graphical designs},
   JOURNAL = {SIAM J. Discrete Math.},
  FJOURNAL = {SIAM Journal on Discrete Mathematics},
    VOLUME = {38},
      YEAR = {2024},
    NUMBER = {1},
     PAGES = {947--964},
      ISSN = {0895-4801,1095-7146},
   MRCLASS = {05C50 (52B12 68Q17 68R10 90C57)},
  MRNUMBER = {4712817},
MRREVIEWER = {Hong\ Zhang},
       DOI = {10.1137/22M1528768},
       URL = {https://doi-org.offcampus.lib.washington.edu/10.1137/22M1528768},
}

@article{chowdhury2025combstructures,
      title={Graphical Designs find Combinatorial Structures}, 
      author={Zawad Chowdhury and Stefan Steinerberger and Rekha R. Thomas},
      year={2025},
      eprint={2507.13327},
      archivePrefix={arXiv},
      journal={arXiv preprint},
      url={https://arxiv.org/abs/2507.13327}, 
}

@article {babecki2022galeduality,
    AUTHOR = {Babecki, Catherine and Thomas, Rekha R.},
     TITLE = {Graphical Designs and Gale Duality},
   JOURNAL = {Math. Program.},
  FJOURNAL = {Mathematical Programming},
    VOLUME = {200},
      YEAR = {2023},
    NUMBER = {2},
     PAGES = {703--737},
      ISSN = {0025-5610,1436-4646},
   MRCLASS = {05C50 (52B35 90C57)},
  MRNUMBER = {4604975},
       DOI = {10.1007/s10107-022-01861-0},
       URL = {https://doi-org.offcampus.lib.washington.edu/10.1007/s10107-022-01861-0},
}

@book {alon2008probabilistic,
    AUTHOR = {Alon, Noga and Spencer, Joel H.},
     TITLE = {The Probabilistic Method},
   EDITION = {Third},
      NOTE = {With an appendix on the life and work of Paul Erd\H os},
 PUBLISHER = {John Wiley \& Sons, Inc., Hoboken, NJ},
      YEAR = {2008},
     PAGES = {xviii+352},
      ISBN = {978-0-470-17020-5},
   MRCLASS = {60-02 (05C80 60C05 60F99 60G42)},
  MRNUMBER = {2437651},
       DOI = {10.1002/9780470277331},
       URL = {https://doi.org/10.1002/9780470277331},
}

@article{steinerberger2020designs,
  title={Generalized Designs on Graphs: Sampling, Spectra, Symmetries},
  author={Steinerberger, Stefan},
  journal={Journal of graph theory},
  volume={93},
  number={2},
  pages={253--267},
  year={2020},
  publisher={Wiley Online Library},
  url={https://arxiv.org/pdf/1803.02235}
}

@article {golubev2020extremal,
    AUTHOR = {Golubev, Konstantin},
     TITLE = {Graphical Designs and Extremal Combinatorics},
   JOURNAL = {Linear Algebra Appl.},
  FJOURNAL = {Linear Algebra and its Applications},
    VOLUME = {604},
      YEAR = {2020},
     PAGES = {490--506},
      ISSN = {0024-3795,1873-1856},
   MRCLASS = {05B99 (05C35 05C50 05C69 05C70 35J05 35P05 35R02)},
  MRNUMBER = {4123763},
       DOI = {10.1016/j.laa.2020.07.012},
       URL = {https://doi.org/10.1016/j.laa.2020.07.012},
}

@article {zhu2023bch,
    AUTHOR = {Zhu, Yan},
     TITLE = {Optimal and Extremal Graphical Designs on Regular Graphs
              associated with Classical Parameters},
   JOURNAL = {Des. Codes Cryptogr.},
  FJOURNAL = {Designs, Codes and Cryptography. An International Journal},
    VOLUME = {91},
      YEAR = {2023},
    NUMBER = {8},
     PAGES = {2737--2754},
      ISSN = {0925-1022,1573-7586},
   MRCLASS = {05B30 (05E30 94B15)},
  MRNUMBER = {4618186},
       DOI = {10.1007/s10623-023-01231-7},
       URL = {https://doi.org/10.1007/s10623-023-01231-7},
}

@article{babecki2021codes,
  title={Codes, Cubes, and Graphical Designs},
  author={Babecki, Catherine},
  journal={Journal of Fourier Analysis and Applications},
  volume={27},
  number={5},
  pages={81},
  year={2021},
  publisher={Springer},
  url={https://link.springer.com/article/10.1007/s00041-021-09852-z}
}

@article{keevash2024existencedesigns,
      title={The Existence of Designs}, 
      author={Peter Keevash},
      year={2014},
      journal={arXiv preprint},
      eprint={1401.3665},
      archivePrefix={arXiv},
      url={https://arxiv.org/abs/1401.3665}, 
}

@article {kuperberg2017combstructures,
    AUTHOR = {Kuperberg, Greg and Lovett, Shachar and Peled, Ron},
     TITLE = {Probabilistic Existence of Regular Combinatorial Structures},
   JOURNAL = {Geom. Funct. Anal.},
  FJOURNAL = {Geometric and Functional Analysis},
    VOLUME = {27},
      YEAR = {2017},
    NUMBER = {4},
     PAGES = {919--972},
      ISSN = {1016-443X,1420-8970},
   MRCLASS = {05D40 (05A15 05B15 05B30 05C65)},
  MRNUMBER = {3678505},
       DOI = {10.1007/s00039-017-0416-9},
       URL = {https://doi.org/10.1007/s00039-017-0416-9},
}

@article {sobolev1962cubature,
    AUTHOR = {Sobolev, S. L.},
     TITLE = {Cubature Formulas on the Sphere which are Invariant under
              Transformations of Finite Rotation Groups},
   JOURNAL = {Dokl. Akad. Nauk SSSR},
  FJOURNAL = {Doklady Akademii Nauk SSSR},
    VOLUME = {146},
      YEAR = {1962},
     PAGES = {310--313},
      ISSN = {0002-3264},
   MRCLASS = {65.55},
  MRNUMBER = {141225},
MRREVIEWER = {A.\ H.\ Stroud},
}

@article {goethals1977sphericaldesigns,
    AUTHOR = {Delsarte, P. and Goethals, J. M. and Seidel, J. J.},
     TITLE = {Spherical Codes and Designs},
   JOURNAL = {Geometriae Dedicata},
  FJOURNAL = {Geometriae Dedicata},
    VOLUME = {6},
      YEAR = {1977},
    NUMBER = {3},
     PAGES = {363--388},
   MRCLASS = {05B99},
  MRNUMBER = {485471},
MRREVIEWER = {Michel\ Deza},
       DOI = {10.1007/bf03187604},
       URL = {https://doi.org/10.1007/bf03187604},
}

@article {bondarenko2013spherical,
    AUTHOR = {Bondarenko, Andriy and Radchenko, Danylo and Viazovska,
              Maryna},
     TITLE = {Optimal Asymptotic Bounds for Spherical Designs},
   JOURNAL = {Ann. of Math. (2)},
  FJOURNAL = {Annals of Mathematics. Second Series},
    VOLUME = {178},
      YEAR = {2013},
    NUMBER = {2},
     PAGES = {443--452},
      ISSN = {0003-486X,1939-8980},
   MRCLASS = {41A63 (41A55 52C35 65C10)},
  MRNUMBER = {3071504},
MRREVIEWER = {Hiroshi\ Nozaki},
       DOI = {10.4007/annals.2013.178.2.2},
       URL = {https://doi-org.offcampus.lib.washington.edu/10.4007/annals.2013.178.2.2},
}

@article {yu2021irredcharpoly,
    AUTHOR = {Yu, Qian and Liu, Fenjin and Zhang, Hao and Heng, Ziling},
     TITLE = {Note on graphs with irreducible characteristic polynomials},
   JOURNAL = {Linear Algebra Appl.},
  FJOURNAL = {Linear Algebra and its Applications},
    VOLUME = {629},
      YEAR = {2021},
     PAGES = {72--86},
      ISSN = {0024-3795,1873-1856},
   MRCLASS = {05C50 (05C31 11R09 12F05)},
  MRNUMBER = {4293741},
       DOI = {10.1016/j.laa.2021.07.013},
       URL = {https://doi-org.offcampus.lib.washington.edu/10.1016/j.laa.2021.07.013},
}

@incollection {chvatal1977threshold,
    AUTHOR = {Chv\'atal, V\'aclav and Hammer, Peter L.},
     TITLE = {Aggregation of inequalities in integer programming},
 BOOKTITLE = {Studies in integer programming ({P}roc. {W}orkshop, {B}onn,
              1975)},
    SERIES = {Ann. Discrete Math.},
    VOLUME = {Vol. 1},
     PAGES = {145--162},
 PUBLISHER = {North-Holland, Amsterdam-New York-Oxford},
      YEAR = {1977},
   MRCLASS = {90C10},
  MRNUMBER = {479384},
MRREVIEWER = {Csaba\ Fabian},
}

@book {mahadevpeled1995threshold,
    AUTHOR = {Mahadev, N. V. R. and Peled, U. N.},
     TITLE = {Threshold graphs and related topics},
    SERIES = {Annals of Discrete Mathematics},
    VOLUME = {56},
 PUBLISHER = {North-Holland Publishing Co., Amsterdam},
      YEAR = {1995},
     PAGES = {xiv+543},
      ISBN = {0-444-89287-7},
   MRCLASS = {05-02 (05C30 05C75)},
  MRNUMBER = {1417258},
MRREVIEWER = {Arkadzi\ Charniak},
}

@article {macharete2024thresholdeigenbasis,
    AUTHOR = {Macharete, Rafael R. and Del-Vecchio, Renata R. and Teixeira,
              Heber and de Lima, Leonardo},
     TITLE = {A {L}aplacian eigenbasis for threshold graphs},
   JOURNAL = {Spec. Matrices},
  FJOURNAL = {Special Matrices},
    VOLUME = {12},
      YEAR = {2024},
     PAGES = {Paper No. 20240029, 13},
      ISSN = {2300-7451},
   MRCLASS = {05C50},
  MRNUMBER = {4815530},
       DOI = {10.1515/spma-2024-0029},
       URL = {https://doi-org.offcampus.lib.washington.edu/10.1515/spma-2024-0029},
}

@misc{borg2026thresholdeigenvectors,
      title={A Note on the Laplacian Eigenvectors of Threshold Graphs}, 
      author={James L. Borg and Irene Sciriha and Zoia Sherman},
      year={2026},
      eprint={2605.03645},
      archivePrefix={arXiv},
      primaryClass={math.CO},
      url={https://arxiv.org/abs/2605.03645}, 
}

@article {guo2005eigenvalue,
    AUTHOR = {Guo, Ji-Ming},
     TITLE = {On the second largest {L}aplacian eigenvalue of trees},
   JOURNAL = {Linear Algebra Appl.},
  FJOURNAL = {Linear Algebra and its Applications},
    VOLUME = {404},
      YEAR = {2005},
     PAGES = {251--261},
      ISSN = {0024-3795,1873-1856},
   MRCLASS = {05C50 (15A18)},
  MRNUMBER = {2149662},
MRREVIEWER = {Weigen\ Yan},
       DOI = {10.1016/j.laa.2005.02.031},
}

@article {steinerbergerthomas2025randomwalks,
    AUTHOR = {Steinerberger, Stefan and Thomas, Rekha R.},
     TITLE = {Random walks, equidistribution and graphical designs},
   JOURNAL = {Adv. in Appl. Math.},
  FJOURNAL = {Advances in Applied Mathematics},
    VOLUME = {165},
      YEAR = {2025},
     PAGES = {Paper No. 102837, 11},
      ISSN = {0196-8858,1090-2074},
   MRCLASS = {05C48 (05C81)},
  MRNUMBER = {4848395},
MRREVIEWER = {Hong\ Zhang},
       DOI = {10.1016/j.aam.2024.102837},
       URL = {https://doi-org.offcampus.lib.washington.edu/10.1016/j.aam.2024.102837},
}

@article {seymourzaslavsky1984spericaldesign,
    AUTHOR = {Seymour, P. D. and Zaslavsky, Thomas},
     TITLE = {Averaging sets: a generalization of mean values and spherical
              designs},
   JOURNAL = {Adv. in Math.},
  FJOURNAL = {Advances in Mathematics},
    VOLUME = {52},
      YEAR = {1984},
    NUMBER = {3},
     PAGES = {213--240},
      ISSN = {0001-8708},
   MRCLASS = {05B30 (26B15)},
  MRNUMBER = {744857},
MRREVIEWER = {J.\ J.\ Seidel},
       DOI = {10.1016/0001-8708(84)90022-7},
       URL = {https://doi.org/10.1016/0001-8708(84)90022-7},
}

@book{hedayat1999orthogonal,
  title={Orthogonal arrays: theory and applications},
  author={Hedayat, A. S. and Sloane, N. J. A. and Stufken, John},
  volume={68},
  year={1999},
  publisher={Springer New York}
}

@article {ferberjainsahsawhney2023irreducibility,
    AUTHOR = {Ferber, Asaf and Jain, Vishesh and Sah, Ashwin and Sawhney,
              Mehtaab},
     TITLE = {Random symmetric matrices: rank distribution and
              irreducibility of the characteristic polynomial},
   JOURNAL = {Math. Proc. Cambridge Philos. Soc.},
  FJOURNAL = {Mathematical Proceedings of the Cambridge Philosophical
              Society},
    VOLUME = {174},
      YEAR = {2023},
    NUMBER = {2},
     PAGES = {233--246},
      ISSN = {0305-0041,1469-8064},
   MRCLASS = {15B52 (11M50 60B20)},
  MRNUMBER = {4545205},
MRREVIEWER = {Yifeng\ Huang},
       DOI = {10.1017/S0305004122000226},
       URL = {https://doi-org.offcampus.lib.washington.edu/10.1017/S0305004122000226},
}

@book{garey1979computers,
  title={Computers and Intractability: A Guide to the Theory of NP-Completeness},
  author={Garey, Michael R. and Johnson, David S.},
  year={1979},
  publisher={W. H. Freeman and Company}
}

\end{document}